\documentclass[runningheads,numbook]{svjour3}                   
\smartqed  
\usepackage{graphicx}
\usepackage{amsmath}
\usepackage{epstopdf} 
\usepackage{mathptmx}      
\usepackage{array}
\usepackage{bigstrut}
\usepackage{epsfig}
\usepackage{amssymb}
\usepackage{rotating}
\usepackage{rotate}
\usepackage{longtable}
\usepackage{amsmath}
\usepackage{color}
\usepackage{float}
\usepackage{lineno}
\usepackage{latexsym,amscd,amsfonts,enumerate,supertabular}
\usepackage{tabularx}
\usepackage{multirow}
\usepackage[numbers,sort]{natbib}

\newcommand{\q}{\quad}
\newcommand{\ee}{{\rm e}\hspace{1pt}}

\journalname{BIT}

\begin{document}

\title{Efficient exponential Runge--Kutta methods of high order: construction and implementation
\thanks{This work has been supported in part by National Science Foundation through award NSF DMS--2012022. 
}
}

\titlerunning{Efficient exponential Runge--Kutta methods of high order}        

\author{Vu Thai Luan\\
\\
\textnormal{
 \centerline{\emph{Dedicated to Professor Alexander Ostermann on the occasion of his 60th birthday.}}
}
}
\authorrunning{Vu Thai Luan} 

\institute{Vu Thai Luan \at
             Department of Mathematics and Statistics, Mississippi State University \\
              410 Allen Hall, 175 President's Circle,
Mississippi State, MS, 39762\\
              \email{luan@math.msstate.edu}           
}

\date{
}

\maketitle

\begin{abstract}
Exponential Runge--Kutta methods have shown to be competitive for the time integration of stiff semilinear parabolic PDEs.  The current construction of stiffly accurate exponential Runge--Kutta methods, however, relies on a convergence result that requires weakening many of the order conditions, resulting in schemes whose stages must be implemented in a sequential way. 
In this work, after showing a stronger convergence result, we are able to derive two new families of fourth- and fifth-order exponential Runge--Kutta methods, which, in contrast to the existing methods, have multiple stages that are independent of one another and share the same format, thereby allowing them to be implemented in parallel or simultaneously, and making the methods to behave like using with much less stages. Moreover, all of their stages involve only one linear combination of the product of $\varphi$-functions (using the same argument) with vectors.
Overall, these features make these new methods to be much more efficient to implement when compared to the existing methods of the same orders. Numerical experiments on a one-dimensional semilinear parabolic problem, a
nonlinear Schrödinger equation, 
and a two-dimensional Gray--Scott model are given  to confirm the accuracy and efficiency of the two newly constructed methods.

\end{abstract}
\section{Introduction}
\label{intro}
In this paper, we are concerned with the construction and implementation of new efficient exponential Runge--Kutta  integrators for solving stiff parabolic PDEs. These PDEs, upon their spatial discretizations, can be cast in the form of semilinear problems
\begin{equation} \label{eq1.1}
u'(t)= A u(t) + g(t, u(t))=F(t, u(t)), \qquad u(t_0)=u_0,
\end{equation}
where the linear part $Au$ usually causes stiffness. The nonlinearity $g(t,u)$ is assumed to satisfy a local Lipschitz condition 
 in a strip along the exact solution.

Exponential Runge--Kutta methods are a popular class of exponential integrators \cite{HO10}, which have shown a great promise as an alternative to standard time integration solvers for  stiff systems and applications in recent years, see e.g. \cite{HO05,LO12b,LO14b,LO13,LO14a,Luan2014,LO16,L17,Michels2017,Ju2017,LD18,Luan18,Pieper2019a}.
The main idea behind these methods is to solve the linear portion of \eqref{eq1.1}  exactly and integrate the remaining nonlinear portion explicitly based on a representation of the exact solution using the variation-of-constants formula.

A $s$-stage explicit exponential Runge--Kutta  (expRK) method \cite{HO05} applied to  \eqref{eq1.1} can be reformulated (see \cite{LO12b,LO14b}) as 
\begin{subequations} \label{eq:expRK}
\begin{align}
 U_{ni}&= u_n + c_i h \varphi _{1} ( c_i h A)F(t_n, u_n) +
 h \sum_{j=2}^{i-1}a_{ij}(h A) D_{nj}, \  2\leq i\leq s,  \label{eq2.1a} \\
u_{n+1}& = u_n + h \varphi _{1} ( h A)F(t_n, u_n) + h \sum_{i=2}^{s}b_{i}(h A) D_{ni}  \label{eq2.1b},
\end{align}
\end{subequations}
where 
\begin{equation} \label{eq1.2}
 D_{ni}= g (t_n+c_i h, U_{ni})- g(t_n, u_n ), \qquad  2\leq i\leq s.
\end{equation}
Here, $U_{ni}$ denote the internal stages that approximate $u(t_n+c_i h)$ using the time step size $h = t_{n+1}-t_n >0$ and nodes $c_i $.  By construction, the coefficients $a_{ij}(z)$ and $b_i (z)$ are usually linear combinations of the entire functions  
\begin{equation} \label{eq1.3}
\varphi_{k}(z)=\int_{0}^{1} \ee^{(1-\theta )z} \frac{\theta ^{k-1}}{(k-1)!}\,\text{d}\theta , \quad k\geq 1
\end{equation}
and their scaled versions $\varphi_{k}(c_i z)$.

A common approach that has been used to determine the unknown matrix functions $a_{ij}(h A)$ and $b_{i}(h A)$ is to expand them as 
$a_{ij}(h A)=\sum_{k\ge 0} \alpha^{(k)}_{ij} (hA)^k $, \ 
$b_{i}(h A)=\sum_{k\ge 0} \beta^{(k)}_{i} (hA)^k
 $
(e.g. using classical Taylor series expansions) to obtain order conditions. Clearly, the boundedness of the remainder terms of these expansions (and thus the error terms) are dependent of $\|A\|$.  Due to stability reasons, such resulting methods  might not be suitable for integrating stiff PDEs, which $A$ typically has a large norm or is even unbounded operator. These methods are thus usually referred as classical (non-stiffly accurate) expRK methods.  
 Unlike this approach, in a seminal contribution \cite{HO05}, Hochbruck and Ostermann derived a new error expansion with the remainder terms that are bounded independently of the stiffness (i.e. not involving the powers of $A$), leading to stiff order conditions, which give rise to the construction of stiffly accurate expRK methods of orders up to four. Following this, in \cite{LO13} Luan and Ostermann developed a systematic theory of deriving stiff order conditions for expRK methods of arbitrary order, thereby allowing the construction of a  fifth-order method in \cite{LO14b}.

In view of the existing stiffly accurate expRK methods in the literature, 
we observe that they were derived  based on a convergence result that requires weakening many of the stiff order conditions (in order to minimize the number of required stages $s$ and matrix functions used in each internal stages $U_{ni}$). As a result, their structures contain internal stages $U_{ni}$ that are dependent of the preceding stages, implying that such methods must be implemented by computing each of these stages sequentially.  Also, the very last stages usually involve several different linear combinations of  $\varphi_{k}(c_i hA)$-functions (using different nodes $c_i$ in their arguments) acting on different sets of vectors.
This would introduce additional computational effort for these stages. For more details, we refer to Sections~\ref{sec2} and \ref{sec:implementation}.

Motivated by the observations above, in this work we show a stronger convergence result for expRK methods up to order five which requires weakening only one order condition (thereby could improve the stability and accuracy) and offers more degree of freedoms in solving order conditions. Using this result and inspired by our recent algorithm, $\mathtt{phipm\_simul\_iom}$, proposed in \cite{Luan18} (which allows one to simultaneously compute multiple linear combinations of $\varphi$- functions acting on a same set of vectors), we construct new methods of orders 4 and 5 which involve only one linear combination of $\varphi$- functions for each stage and have multiple internal stages $U_{ni}$ that are independent of one another, thereby allowing them to be computed in parallel. Furthermore, one can derive these independent stages in a way that they share the same form of linear combination of $\varphi_{k}(c_i hA)$- functions acting on the same set of vectors, allowing them to be implemented simultaneously (by one evaluation). 
While these independent states  can be computed in parallel (as mentioned above) by any algorithm which approximates the action of (the linear combination of) $\varphi$- functions, we note that the possibility to compute them simultaneously is a new feature that can be used with our algorithm $\mathtt{phipm\_simul\_iom}$ (other algorithms, e.g., that do not require the construction of Krylov subspaces, might not support computing these stages simultaneously).
Overall,  this makes the new methods to behave like methods using much less number of stages (even when compared to the existing methods of the same orders), meaning that they require much less number of evaluations for linear combinations of $\varphi$- functions, and are thus more efficient.

The paper is organized as follows. In Section~\ref{sec2}, we describe our motivation, propose new ideas, and review the existing expRK methods in the literature with respect to these ideas. Following this, in Section~\ref{sec3} we prove a stronger convergence result (Theorem~\ref{th3.1}) for expRK methods, which requires relaxing only one order condition.
This allows us to construct more efficient methods  in Section~\ref{sec4}. In particular, we are able to derive two new families of fourth- and fifth- order stiffly accurate expRK methods called  $\mathtt{expRK4s6}$ (4th-order 6-stage but requires 4 independent stage evaluation only) and  $\mathtt{expRK5s10}$ (5th-order 10-stage but requires 5 independent stage evaluation only), respectively. In Section~\ref{sec:implementation}, we present details implementation of these two new methods, as well as the existing stiffly accurate expRK schemes of the same orders (for comparison). In the last section, numerical examples including one and two-dimensional stiff PDEs are presented to demonstrate the accuracy and efficiency of the two newly constructed expRK integrators.

\section{Motivation and existing methods} \label{sec2}
In this section, we start our motivation by taking a closer look at an efficient way for implementing expRK methods  \eqref{eq:expRK}. Then, we propose some ideas to derive more efficient methods with respect to this efficient implementation along with reviewing the current methods. 
\subsection{An efficient way of implementation} \label{sec2.1}
Clearly, each stage ($U_{ni}$ or $u_{n+1}$) of \eqref{eq:expRK} requires computing matrix functions of the form $\varphi_k(c_i h A)v_k$ ($0 < c_i \le 1$), where $v_k$ is some vector (could be $F(t_n, u_n), D_{ni}$ or a linear combination of these). 
Thanks to recent developments \cite{AH11,NW12,CKOS16,GaudreaultPudykiewicz16}, one can efficiently compute a linear combination of  $\varphi$-functions acting on a set of input vectors $V_0,\ldots,V_q$
\begin{equation} \label{eq2.1}
\varphi_0 (M)V_0 + \varphi_1 (M)V_1 +\varphi_2 (M)V_2+\cdots+\varphi_{q}(M)V_{q},
\end{equation}
where $M$ is some square matrix. This is crucial when implementing exponential integrators. Very recently, in \cite{Luan18}, we were able to improve the implementations presented in \cite{NW12,GaudreaultPudykiewicz16}, resulting in the routine $\mathtt{phipm\_simul\_iom}$.
The underlying method in this algorithm is the use of an adaptive time-stepping technique combined with Krylov subspace methods, which allows us to simultaneously compute multiple linear combinations of type \eqref{eq2.1} using different scaling factors $\rho_1, \cdots, \rho_r$ of $M$, i.e.,
\begin{equation} \label{eq2.2}
\begin{aligned}
\varphi_0 (\rho_1 M)V_0 + \varphi_1 (\rho_1 M)V_1 &+\varphi_2 (\rho_1 M)V_2+\cdots+\varphi_{N}(\rho_1 M)V_{q},\\
& \vdots\\
\varphi_0 (\rho_r M)V_0 + \varphi_1 (\rho_r M)V_1 &+\varphi_2 (\rho_r M)V_2+\cdots+\varphi_{N}(\rho_r M)V_{q}.
\end{aligned}
\end{equation}
Now taking $M=hA$ and considering  $\rho_k$  ($1 \le k \le r$)  as nodes $c_i$ used in expRK methods immediately suggests that one can compute the following ($s-1$) linear linear combinations 
\begin{equation}
    \label{eq:phi_linear_combination}
\varphi_1(c_i h A)V_1 + \varphi_2(c_i h A)V_2+ \ldots + \varphi_N(c_i h A) V_q, \q 2 \le i \le s
\end{equation}
simultaneously 
by using only one evaluation (i.e., one call to $\mathtt{phipm\_simul\_iom}$). Note that this requires the use of a \emph{same} set of  vectors $[V_1,\ldots,V_q ]$ for all the linear combinations in~\eqref{eq:phi_linear_combination}.

Motivated by this, we see that if a $s$-stage expRK scheme  \eqref{eq:expRK}  is constructed in such a way that each internal stage $U_{ni}$ has the form
\begin{equation}
    \label{eq:Uni_linear_combination}
U_{ni}=u_n+ \varphi_1(c_i h A)V_{1i} + \varphi_2(c_i h A)V_{2i}+ \ldots + \varphi_N(c_i h A) V_{qi},
\end{equation}
which includes only one linear combination of $\varphi$- functions using \emph{exactly} node $c_i$ as an argument in all $\varphi_k$ functions, then the scheme will contain a total of $s$ such linear combinations ($s-1$  for $U_{ni}$ and 1 for $u_{n+1}$ as \eqref{eq2.1b} can be always written in the form of \eqref{eq:Uni_linear_combination} with $c_i=1$), thereby requiring $s$ evaluations only. Furthermore, since the sets of vectors $[V_{1i}, V_{2i}, \cdots, V_{qi}]$ in  \eqref{eq:Uni_linear_combination} are usually different for each $U_{ni}$, \eqref{eq:phi_linear_combination} also suggests that the efficiency  will be significantly increased if one could build such stages (or a group of)  $U_{ni}$ of the form \eqref{eq:Uni_linear_combination} that share the same format
(i.e., having the same set of acting vectors $V_{1i} \equiv V_1,\ldots,V_{qi} \equiv V_q$) or that are independent of one another. 
As this allows to compute such stages simultaneously by one evaluation or to implement them in parallel similarly to our construction of parallel exponential Rosenbrock methods \cite{LO16}), it certainly reduces the total number of required evaluations and thus speedups the computing time.

With respect to these observations, we now review the existing expRK schemes in the literature. Since our focus is on stiff problems, we will discuss only on stiffly accurate expRK methods, meaning that they satisfy the stiff order conditions (see  Section~\ref{sec3} below).
\subsection{Existing schemes and remarks} \label{sec2.2}
In \cite{HO05},  expRK methods of orders up to four have been derived. 
For later reference, we name the second-order, the third-order, and the fourth-order methods in that work  as $\mathtt{expRK2s2}$,  $\mathtt{expRK3s3}$, and $\mathtt{expRK4s5}$, respectively.
In \cite{LO14b},  we have constructed an expRK method of order five called $\mathtt{expRK5s8}$. To discuss all of these schemes in terms of the implementation, we rewrite their internal stages $U_{ni}$ and $u_{n+1}$ as linear combinations of $\varphi$- functions like \eqref{eq:Uni_linear_combination} and display them as follows (Note that since the first-order method, the exponential Euler scheme
$
u_{n+1}= u_n +  \varphi_1 ( h A) hF(t_n,u_n),
$
has no internal stage, we do not consider it here).
 
\noindent $\mathtt{expRK2s2}$: 
\begin{equation}\label{eq:expRK2s2}
\begin{aligned}
U_{n2}=u_n &+\varphi_1 (c_2 h A)  c_2 hF(t_n,u_n), \\
u_{n+1}= u_n &+  \varphi_1 ( h A) hF(t_n,u_n) + \varphi_2 (h A)\tfrac{1}{c_2} h D_{n2}.
\end{aligned}
\end{equation} 
\noindent $\mathtt{expRK3s3}$  (a representative with $c_2 \ne \tfrac{2}{3}$):
\begin{equation}\label{eq:expRK3s3}
\begin{aligned}
U_{n2}=u_n &+\varphi_1 (c_2 h A)  c_2 hF(t_n,u_n), \\
U_{n3}=u_n &+ \varphi_1 (\tfrac{2}{3} h A)\tfrac{2}{3} hF(t_n, u_n)+ \varphi_2 (\tfrac{2}{3} h A) \tfrac{4}{9 c_2}h D_{n2}, \\
u_{n+1}= u_n &+  \varphi_1 ( h A) hF(t_n,u_n) + \varphi_2 (h A)\tfrac{3}{2} h D_{n2}.
\end{aligned}
\end{equation} 
\noindent $\mathtt{expRK4s5}$ (the only existing fourth-order stiffly accurate expRK method constructed by Hochbruck and Ostermann \cite{HO05}):
\begin{equation}\label{eq:expRK4s5}
\begin{aligned}
U_{n2}=u_n &+\varphi_1 (\tfrac{1}{2} h A)  \tfrac{1}{2} hF(t_n,u_n), \\
U_{n3}=u_n &+ \varphi_1 (\tfrac{1}{2} h A)\tfrac{1}{2} hF(t_n, u_n)+ \varphi_2 (\tfrac{1}{2} h A) h D_{n2}, \\
U_{n4}=u_n &+  \varphi_1 ( h A) hF(t_n,u_n)+ \varphi_2 (h A) h(D_{n2} + D_{n3}),\\
U_{n5}=u_n &+ [\varphi_1 ( \tfrac{1}{2} h A) \tfrac{1}{2} hF(t_n, u_n)
                      + \varphi_2 ( \tfrac{1}{2} h A) \tfrac{1}{4} h (2 D_{n2} + 2 D_{n3} - D_{n4})\\
                     & + \varphi_3 ( \tfrac{1}{2} h A) \tfrac{1}{2} h (-D_{n2}  - D_{n3} + D_{n4})]
                        + [\varphi_2 (h A) \tfrac{1}{4} h (D_{n2} + D_{n3} - D_{n4})\\
                     &+  \varphi_3 (h A)  h (-D_{n2} - D_{n3} + D_{n4})],\\
u_{n+1}= u_n &+  \varphi_1 ( h A) hF(t_n,u_n) +  \varphi_2 (h A) h (-D_{n4}+4 D_{n5})+  \varphi_3 (h A) h (4D_{n4} -8 D_{n5}).                    
\end{aligned}
\end{equation} 

\noindent $\mathtt{expRK5s8}$ (the only existing fifth-order stiffly accurate expRK method constructed by Luan and Ostermann \cite{LO14b}):
\begin{equation}\label{eq:expRK5s8}
\begin{aligned} \notag
U_{n2}=u_n &+\varphi_1 (\tfrac{1}{2} h A)  \tfrac{1}{2} hF(t_n,u_n), \\
U_{n3}=u_n &+ \varphi_1 (\tfrac{1}{2} h A)\tfrac{1}{2} hF(t_n, u_n)+ \varphi_2 (\tfrac{1}{2} h A)  \tfrac{1}{2} h D_{n2}, \\
U_{n4}=u_n &+  \varphi_1 ( \tfrac{1}{4} h A)  \tfrac{1}{4} hF(t_n,u_n)+ \varphi_2 ( \tfrac{1}{4} h A) \tfrac{1}{8}  h D_{n3},\\
U_{n5}=u_n &+ \varphi_1 ( \tfrac{1}{2} h A) \tfrac{1}{2} hF(t_n, u_n)
                      + \varphi_2 ( \tfrac{1}{2} h A) \tfrac{1}{2} h (- D_{n3} + 4D_{n4})
                      + \varphi_3 ( \tfrac{1}{2} h A) h (2D_{n3} - 4D_{n4})\\
U_{n6}=u_n &+ \varphi_1 ( \tfrac{1}{5} h A) \tfrac{1}{5} hF(t_n, u_n)
                      + \varphi_2 ( \tfrac{1}{5} h A) \tfrac{1}{25} h (8D_{n4} - 2D_{n5})
                      + \varphi_3 ( \tfrac{1}{5} h A)  \tfrac{1}{125}h (-32 D_{n4} + 16 D_{n5}),\\
U_{n7}=u_n &+[ \varphi_1 ( \tfrac{2}{3} h A) \tfrac{2}{3} hF(t_n, u_n)
                      + \varphi_2 ( \tfrac{2}{3} h A) h (\tfrac{-16}{27} D_{n5} +\tfrac{100}{27}  D_{n6})
                      + \varphi_3 ( \tfrac{2}{3} h A) h (\tfrac{320}{81} D_{n5} -\tfrac{800}{81}  D_{n6})]\\
                   &+[\varphi_2 ( \tfrac{1}{5} h A) h (\tfrac{-20}{81} D_{n4} +\tfrac{5}{243}  D_{n5}+\tfrac{125}{486}  D_{n6})
                     +\varphi_3 ( \tfrac{1}{5} h A)  h (\tfrac{16}{81} D_{n4} -\tfrac{4}{243}  D_{n5}-\tfrac{50}{243}  D_{n6})],  \\
U_{n8}=u_n &+\big[ \varphi_1 ( h A) hF(t_n, u_n)
                      + \varphi_2 ( h A) h (\tfrac{-16}{3} D_{n5} +\tfrac{250}{21}  D_{n6}+\tfrac{27}{14}  D_{n7})\\
                      &+ \varphi_3 ( h A) h (\tfrac{208}{3} D_{n5} -\tfrac{250}{3}  D_{n6}- 27 D_{n7})
                        + \varphi_4 ( h A) h (-240D_{n5} + \tfrac{1500}{7}  D_{n6}+  \tfrac{810}{7} D_{n7})\big]\\
                     &+ \big[\varphi_2 ( \tfrac{1}{5} h A) h (\tfrac{-4}{7} D_{n5} + \tfrac{25}{49}  D_{n6} +\tfrac{27}{98} D_{n7}) 
                     + \varphi_3 ( \tfrac{1}{5} h A) h (\tfrac{8}{5} D_{n5} - \tfrac{10}{7}  D_{n6} - \tfrac{27}{35} D_{n7})\\ 
                  &+\varphi_4 ( \tfrac{1}{5} h A) h (\tfrac{-48}{35} D_{n5} + \tfrac{60}{49}  D_{n6} + \tfrac{162}{245} D_{n7})\big]\\
                  &+ \big[\varphi_2 ( \tfrac{2}{3} h A) h (\tfrac{-288}{35} D_{n5} + \tfrac{360}{49}  D_{n6} +\tfrac{972}{245} D_{n7}) 
                     + \varphi_3 ( \tfrac{2}{3} h A) h (\tfrac{384}{5} D_{n5} - \tfrac{480}{7}  D_{n6} - \tfrac{1296}{35} D_{n7})\\ 
                  &+\varphi_4 ( \tfrac{2}{3} h A) h (\tfrac{-1536}{7} D_{n5} + \tfrac{9600}{49}  D_{n6} + \tfrac{5184}{49} D_{n7})\big],\\
u_{n+1}= u_n &+  \varphi_1 ( h A) hF(t_n,u_n) +  \varphi_2 (h A) h ( \tfrac{125}{14} D_{n6} -\tfrac{27}{14} D_{n7} +  \tfrac{1}{2} D_{n8}) \\
  &+  \varphi_3 (h A) h ( \tfrac{-625}{14} D_{n6} + \tfrac{162}{7} D_{n7} -  \tfrac{13}{2} D_{n8})     
    +    \varphi_4 (h A) h ( \tfrac{1125}{14} D_{n6} -  \tfrac{405}{7} D_{n7}  +  \tfrac{45}{2} D_{n8}).                
\end{aligned}
\end{equation} 
\begin{remark}\label{remark2.1}
In view of the structures of these schemes, one can see that only the second- and third-oder schemes ($\mathtt{expRK2s2}$, $\mathtt{expRK3s3}$) have all 
 $U_{ni}$ in the form  \eqref{eq:Uni_linear_combination}. While $\mathtt{expRK2s2}$ requires one internal stage $U_{n2}$, $\mathtt{expRK3s3}$ needs two internal stages with $U_{n3}$ depends on  $U_{n2}$, making these stages cannot be computed simultaneously.
As for $\mathtt{expRK4s5}$, to the best of our knowledge, this 5-stage scheme is the only existing fourth-order stiffly accurate expRK method. As seen, among  its internal stages the three internal stages $U_{n2}$, $U_{n3}$, and $U_{n4}$ are of the form 
\eqref{eq:Uni_linear_combination} but again their corresponding sets of vectors $[V_{ki}]$ are not the same ($[V_{k2}]=[\tfrac{1}{2} hF(t_n,u_n)], [V_{k3}]=[\tfrac{1}{2} hF(t_n,u_n), hD_{n2} ],  [V_{k4}]=[hF(t_n,u_n), h(D_{n2}+D_{n3})]$),
 and because of \eqref{eq1.2}, they are not independent of one another ($U_{n4}$ and $U_{n3}$ depend on their preceding stages). Therefore one needs 3 sequential evaluations for computing these three stages. Also, we note that the last internal stage $U_{n5}$ depends on all of its preceding stages and involves two different linear combinations of $\varphi_k$- functions with different scaling factors $c_5=\tfrac{1}{2}$ and $c_4=1$, namely,  $\sum_{k} \varphi_k (\tfrac{1}{2} hA)V_k$ and $\sum_{k}\varphi_k (hA)W_k$ (grouped in two different brackets $[ \ ]$), which has to be implemented by 2 separate evaluations. The final stage $u_{n+1}$ depends on $U_{n4}$ and $U_{n5}$. As a result, this scheme must be implemented in a sequential way, which requires totally 6 evaluations for 6 different linear combinations.
Similarly, to the best of our knowledge, $\mathtt{expRK5s8}$ is also the only existing fifth-order stiffly accurate expRK methods. From the construction of this scheme  \cite{LO14b}, one needs 8 stages. Among them, the first five internal stages are of the form \eqref{eq:Uni_linear_combination}. We note, however, that the last two internal stages $U_{n7}$ and $U_{n8}$ involves 2 and 3 different linear combinations (grouped in different  brackets $[ \ ]$) of $\varphi_k$- functions (with different scaling factors) acting on different sets of vectors. And each stage ($U_{ni}$ or $u_{n+1}$) depends on all the preceding stages (except for the first stage $U_{n2}$). Thus, this scheme must be also implemented in a sequential way (it also does not have any group of internal stages that can be computed simultaneously). Clearly,  it requires totally  11 evaluations (11 different linear combinations of $\varphi$ functions).
\end{remark}
\begin{remark}\label{remark2.2}
The resulting structures of the expRK schemes discussed in Remark~\ref{remark2.1} can be explained by taking a closer look at their constructions presented in  \cite{HO05,LO14b}. Namely, these methods have been analyzed and derived by using a weakened convergence result, i.e., weakening of many order conditions  in order to minimize the number of required stages $s$ and the number of matrix functions in each internal stage $U_{ni}$.
Specifically, for fourth-order methods (e.g., $\mathtt{expRK4s5}$)  4  out of 9 order conditions have to be relaxed  and for fifth-order methods (e.g., $\mathtt{expRK5s8}$) 9  out of 16 order conditions have to be relaxed. As a trade off, each stage of these methods depends on the preceding stages (thus  the resulting schemes must be implemented by computing each stage sequentially) and the very last stages usually involve different linear combinations of  $\varphi_{k}$-functions (with several different nodes $c_i$ as scaling factors) acting on not the same set of vectors, which then require additional sequential evaluations. For more details, see Section~\ref{sec4} below.
 \end{remark}

\section{Stiff order conditions and convergence analysis}\label{sec3}
Inspired by the motivation and remarks in Section~\ref{sec2}, we next present a stronger convergence result which later allows a construction of new efficient methods of high order. For this, we first recall the stiff order conditions for expRK methods up to order 5  (see \cite{LO12b,LO14b}).
\subsection{Stiff order conditions for methods up to order 5}\label{sec3.1}
Let $\tilde{e}_{n+1}=\hat{u}_{n+1}- u(t_{n+1})$ denote the local error of \eqref{eq:expRK},  i.e., the difference between the numerical solution $\hat{u}_{n+1}$ obtained by \eqref{eq:expRK} after one step starting from the `initial condition' $u(t_n)$ and the corresponding exact solution $u(t_{n+1})$ of \eqref{eq1.1} at $t_{n+1}$. 

To simplify the notation in this section we set $f(t)=g(t, u(t))$  as done in \cite{HO05}, and additionally denote
 $G_{k,n}=D^{k} g(t_n, u(t_n))$ be the $k$-th partial Fr\'echet  derivative (with respect to $u$) evaluated at $u(t_n)$.
 Our results in \cite{LO12b} (Sect. 4.2) or \cite{LO13} (Sect. 5.1) showed that
\begin{equation}\label{eq3.1}
\begin{aligned}
\tilde{e}_{n+1}&=h^2\psi _{2} ( h A)\,f'(t_n)+ h^3\psi _{3} ( h A) \,f''(t_n)
+  h^4 \psi_{4} ( h A) \,f'''(t_n) + h^5 \psi _{5} ( h A) \,f^{(4)}(t_n) \\
&\quad +\mathbf{R}_n+ \mathcal{O}(h^6)
\end{aligned}
\end{equation}
with  the remaining terms
\begin{align} 
\mathbf{R}_n &= h^3 \sum_{i=2}^{s}b_{i} G_{1,n} \psi_{2,i} \,f'(t_n)
+ h^4 \sum_{i=2}^{s}b_{i} G_{1,n} \psi_{3,i} \,f''(t_n) 
+ h^4  \sum_{i=2}^{s}b_{i} G_{1,n} \sum_{j=2}^{i-1}a_{ij} G_{1,n} \psi_{2,j} \,f'(t_n)  \notag \\ 
& + h^4 \sum_{i=2}^{s}b_{i} c_i G_{2,n}\big(u'(t_n),
\psi_{2,i} \,f'(t_n) \big) + h^5 \sum_{i=2}^{s}b_{i} G_{1,n}\psi_{4,i}\,f'''(t_n) \notag  \\
& + h^5 \sum_{i=2}^{s}b_{i} G_{1,n} \sum_{j=2}^{i-1}a_{ij} G_{1,n} \psi_{3,j}\,f''(t_n) 
+ h^5 \sum_{i=2}^{s}b_{i} G_{1,n} \sum_{j=2}^{i-1}a_{ij} G_{1,n}\sum_{k=2}^{j-1}a_{jk}
G_{1,n} \psi_{2,k} \,f'(t_n) \label{eq3.2} \\
&+ h^5 \sum_{i=2}^{s}b_{i} G_{1,n} \sum_{j=2}^{i-1}a_{ij}c_j G_{2,n}\big(u'(t_n), \psi_{2,j}
\,f'(t_n)\big)  + h^5 \sum_{i=2}^{s}b_{i} c_i G_{2,n}\big(u'(t_n), \psi_{3,i} \,f''(t_n) \big) \notag  \\
& + h^5 \sum_{i=2}^{s}b_{i} c_i G_{2,n}\big(u'(t_n),\sum_{j=2}^{i-1}a_{ij} G_{1,n} \psi_{2,j}
\,f'(t_n) \big) 
+ h^5 \sum_{i=2}^{s} \frac{b_{i}}{2!} G_{2,n}\big(\psi_{2,i} \,f'(t_n),
\psi_{2,i} \,f'(t_n) \big) \notag  \\
& + h^5 \sum_{i=2}^{s}b_{i} \frac{c^2_i}{2!} G_{2,n}\big(u''(t_n),\psi_{2,i} \,f'(t_n)\big) 
+ h^5 \sum_{i=2}^{s}b_{i} \frac{c^2_i}{2!} G_{3,n}\big(u'(t_n),
u'(t_n),\psi_{2,i} \,f'(t_n) \big). \notag
\end{align}
Here, (and from now on) we use the abbreviations $a_{ij}=a_{ij}(hA)$, $b_{i}=b_{i}(hA), \varphi_{j,i}=\varphi_{j} (c_i hA)$ and 
\begin{subequations}\label{eq3.3}
\begin{align}
\psi_{j} (hA)&= \sum_{i=2}^{s}b_{i}\frac{c^{j-1}_{i}}{(j-1)!}- \varphi _{j} (hA), \q j\geq 2  \label{eq3.3a}\\
\psi_{j,i}&=\psi_{j,i} (hA)= \sum_{k=2}^{i-1}a_{ik}\frac{c^{j-1}_{k}}{(j-1)!}- c^{j}_{i}\varphi_{j,i}.
\label{eq3.3b}
\end{align}
\end{subequations}
Requiring a local error truncation  $ \tilde{e}_{n+1}=\mathcal{O}(h^6)$ results in the stiff order conditions for methods of order up to 5, which are displayed in Table~\ref{tb1} below.
{
\setlength{\extrarowheight}{1.4 pt}
\renewcommand{\arraystretch}{1.35}
\vspace{-3mm}
\begin{table}[ht]
\begin{center}
\caption{Stiff order conditions for explicit exponential Runge--Kutta methods up to order 5. The variables
$Z$, $J$, $K$, $L$ denote arbitrary square matrices, and $B$ an arbitrary bilinear mapping of appropriate
dimensions. The functions $\psi_l$ and $\psi_{k,l}$ are defined in \eqref{eq3.3}.}\label{tb1}
\normalsize
\vspace{1.4mm}
\begin{tabular}{ |c|c|c| }
\hline
{\ No.\ } & Stiff order condition & {\ Order\ } \\
\hline
1& $\psi_2(Z)=0 \Longleftrightarrow  \sum_{i=2}^{s}b_{i}(Z)c_i = \varphi_2(Z)$& 2 \\
 \hline
2& $\psi_3(Z)=0 \Longleftrightarrow \sum_{i=2}^{s}b_{i}(Z)\frac{c_i^2}{2!} = \varphi_3(Z)$&3 \\
3& $\sum_{i=2}^{s}b_{i}(Z) J \psi _{2,i}(Z)=0 $& 3 \\
 \hline
4&$\psi_4(Z)=0\Longleftrightarrow \sum_{i=2}^{s}b_{i}(Z)\frac{c_i^3}{3!} = \varphi_4(Z)$&4 \\
5& $\sum_{i=2}^{s}b_{i}(Z) J \psi _{3,i}(Z)=0 $& 4 \\
6& $\sum_{i=2}^{s}b_{i}(Z) J \sum_{j=2}^{i-1}a_{ij}(Z) J \psi_{2,j}(Z)=0 $& 4 \\
7 & $\sum_{i=2}^{s}b_{i}(Z) c_i K\psi_{2,i}(Z)=0 $& 4 \\
\hline
8&$\psi_5(Z)=0 \Longleftrightarrow \sum_{i=2}^{s}b_{i}(Z)\frac{c_i^4}{4!} = \varphi_5(Z)$& 5 \\
9&$ \sum_{i=2}^{s}b_{i}(Z) J \psi_{4,i}(Z)=0 $& 5 \\
10&$\sum_{i=2}^{s}b_{i}(Z) J \sum_{j=2}^{i-1}a_{ij}(Z) J \psi_{3,j}(Z)=0 $& 5 \\
11&\quad $\sum_{i=2}^{s}b_{i}(Z) J \sum_{j=2}^{i-1}a_{ij}(Z) J \sum_{k=2}^{j-1}a_{jk}(Z) J \psi_{2,k}(Z)=0\quad$& 5 \\
12&$\sum_{i=2}^{s}b_{i}(Z) J \sum_{j=2}^{i-1}a_{ij}(Z)c_j K \psi_{2,j}(Z) =0 $& 5 \\
13&$\sum_{i=2}^{s}b_{i}(Z) c_i K \psi_{3,i}(Z)=0 $& 5 \\
14&$\sum_{i=2}^{s}b_{i}(Z) c_i K \sum_{j=2}^{i-1}a_{ij}(Z) J \psi_{2,j}(Z) =0 $& 5 \\
15&$\sum_{i=2}^{s}b_{i}(Z)  B \big(\psi_{2,i}(Z) , \psi_{2,i}(Z) \big)=0 $& 5 \\
16& $\sum_{i=2}^{s}b_{i}(Z) c^2_i L \psi_{2,i}(Z)=0 $& 5 \\
\hline
\end{tabular}
\end{center}
\end{table}
}
\subsection{A stronger convergence result}
The convergence analysis of exponential Runge--Kutta methods is usually performed in the framework of analytic semigroups on a Banach space $X$ with the following assumptions (see e.g. \cite{HO05,LO14b}):

{\em Assumption 1. The linear operator $A$ is the infinitesimal generator of an analytic semigroup $\ee^{tA}$ on $X$}. This implies that 
\begin{equation} \label{eq:bound1}
\|\ee^{tA}\|_{X\leftarrow X}\leq C, \quad t\geq 0
\end{equation}
and consequently $\varphi_k(h A)$,  the coefficients $a_{ij}(h A)$ and $b_{i}(h A)$ of the method are bounded operators. Furthermore, the following stability bound (see  \cite[Lemma 1]{HO05})
\begin{equation} \label{eq:bound2}
\left \|h A \sum_{j=1}^{n}\ee^{j h A} \right\|_{X\leftarrow X}  \leq C
\end{equation}
holds uniformly for all $n\geq 1$ and $h>0$ with $0<nh\le T-t_0$.

{
\em Assumption 2 (for high-order methods). The solution $u:[t_0, T]\to X$  of (\ref{eq1.1}) is sufficiently smooth with derivatives in $X$ and $g: [t_0, T] \to X$ is sufficiently often Fr\'echet differentiable in a strip along the exact solution. All occurring derivatives are assumed to be uniformly bounded.}

Let $ e_{n+1} = u_{n+1} - u(t_{n+1})$ denote the global error at time $t_{n+1}$.
 In \cite{LO14b}, we have shown that $e_n$ satisfies the recursion 
\begin{equation} \label{eq3.4}
e_{n}=h\sum_{j=0}^{n-1} \ee^{(n-j)hA} \mathcal{K}_j (e_j)e_j + \sum_{j=0}^{n-1} \ee^{jhA}\tilde{e}_{n-j},
\end{equation}
where $\mathcal{K}_{j} (e_j)$ are bounded operators on $X$.

Motivated by Remark~\ref{remark2.2}, we now give a stronger convergence result (compared to those results given in \cite{HO05,LO14b}) in the sense that it requires relaxing only one order condition.
\begin{theorem}\label{th3.1} (Convergence)
Let the initial value problem \eqref{eq1.1} satisfy Assumptions 1--2. Consider for its numerical solution an explicit exponential Runge--Kutta method \eqref{eq:expRK} that fulfills the order conditions of
Table~\ref{tb1} up to order $p$ ($2 \le p \le 5$) in a strong form with the exception that only one condition $\psi_{p}(Z)=0$ holds in a weakened form, i.e.,  $\psi_{p}(0)=0$. Then, the method is convergent of order~$p$. In particular, the numerical solution $u_n$ satisfies the error bound
\begin{equation}\label{eq3.5}
\| u_n -u(t_n)\|\leq C h^p 
\end{equation}
uniformly on compact time intervals \ $t_0 \leq  t_n =t_0+nh \leq  T$ with a constant $C$ that depends on $T-t_0$, but is independent of $n$ and $h$.
\end{theorem}
\begin{proof}
The proof can be carried out in a very similar way as done in \cite[Theorem 4.2]{LO14b}.
 In view of \eqref{eq3.1} and \eqref{eq3.2} and employing the assumptions of Theorem~\ref{th3.1} on the order conditions, we have  $\mathbf{R}_n=0$ and thus
\begin{equation} \label{eq3.6}
\tilde{e}_{n+1}=h^p \big( \psi_{p} ( h A)-\psi_{p} (0)\big) G_{p-1,n}+h^{p+1} \mathbf{S}_n,
\end{equation}
where 
$G_{p-1,n}$ is defined in Section~\ref{sec3.1} and
$\mathbf{S}_n$ involves the terms multiplying $h^{p+1}$ and higher order  in  \eqref{eq3.1} (clearly, $\|\mathbf{S}_n\|\leq C$).
Inserting \eqref{eq3.6} (with index $n-j-1$ in place of $n$) into \eqref{eq3.4} and using the fact that there exists a bounded operator $\tilde{\psi}_{p}(hA)$ such that  $ \psi_{p} ( h A)-\psi_{p} (0)=\tilde{\psi}_{p}(hA)hA$ yields
\begin{equation} \label{eq3.7}
e_{n}=h\sum_{j=0}^{n-1} \ee^{(n-j)hA} \mathcal{K}_j (e_j)e_j + h^p\sum_{j=0}^{n-1} hA\ee^{jhA}\tilde{\psi}_{p}(hA) G_{p-1,n-j-1} +h^{p+1} \mathbf{S}_{n-j-1}.
\end{equation}
Using \eqref{eq:bound1}, \eqref{eq:bound2} and an application of a discrete Gronwall lemma shows  \eqref{eq3.5}.\qed
\end{proof}
With the result of Theorem~\ref{th3.1} in hand, we are now ready to derive more efficient methods. In particular, we will solve the system of stiff order conditions of  Table~\ref{tb1} in the context of Theorem~\ref{th3.1}. It turns out that for methods of high order this will require an increase in the number of stages $s$. However, we will have more degree of freedoms for constructing our desired methods as seen in Section~\ref{sec4} below. In addition, by relaxing only one order condition, we expect methods resulted from Theorem~\ref{th3.1} to have better stability (and thus may be more accurate) when integrating stiff systems (see Section~\ref{sec:num}).
\section{Derivation of new efficient exponential Runge--Kutta methods}\label{sec4}
In this section, we will derive methods which have the following features: (i) containing multiple internal stages $U_{ni}$ that are independent of each other (henceforth called \emph{parallel stages}) and share the same format (thereby allowing them to be implemented in parallel); (ii) involving less number of evaluations of the form \eqref{eq:Uni_linear_combination} when compared to the existing methods of the same orders (thus behaving like methods that use fewer number of stages $s$).

We first start with methods of order  $p \le 3$. When solving order conditions for these methods (requiring at least $s=2$ and $s=3$ for second- and third-order methods, respectively), one can easily show that it is not possible to fulfill the desired feature (ii), particularly when comparing with  $\mathtt{expRK2s2}$ (order 2, 2-stage) and $\mathtt{expRK3s3}$ (order 3, 3-stage) mentioned in Section~\ref{sec2}. We omit the details.
Therefore, we will focus on the derivation of new methods of higher orders, namely, orders 4 and 5. 

\subsection{A family of fourth-order methods with parallel stages}\label{sec4.1}
Deriving methods of order 4 requires solving the set of 7 stiff order conditions 1--7 in Table~\ref{tb1}. First, we discuss on the required number of stages $s$.
It is shown in \cite[Sect.5.3]{HO05} that $s=5$ is the minimal number of stages required to construct a family of fourth-order methods which satisfies conditions 1--3 in the strong sense and conditions 4--7 in the weakened form (relaxing $b_i(Z)$ as $b_i(0)$). In other words, with $s=5$ it is not possible to fulfill the order conditions in the context of Theorem~\ref{th3.1}, which requires only condition 4 holds in a weakened form $\psi_{4}(0)=0$ or equivalently $\sum_{i=2}^{s}b_{i}(0)\frac{c_i^3}{3!} = \varphi_4(0)=1/24$. Therefore, we consider $s=6$.  In this case, conditions 1, 2, and the weakened condition 4 are
\begin{subequations} \label{eq4.1}
\begin{align}
b_2 c_2+ b_3 c_3+ b_4 c_4+b_5 c_5+b_6 c_6 &=\varphi_2, \\
b_2 c^2_2+ b_3 c^2_3+ b_4 c^2_4+b_5 c^2_5+b_6 c^2_6&=2\varphi_3,\\
b_2 (0) c^3_2+ b_3 (0)c^3_3+ b_4 (0) c^3_4+b_5 (0) c^3_5+b_6 (0) c^3_6&=6\varphi_4 (0)=1/4,
\end{align}
\end{subequations}
 and conditions 3, 5, 7 and 6 are
\begin{subequations} \label{eq4.2}
\begin{align}
&b_2 J \psi_{2,2}+ b_3 J \psi_{2,3}+ b_4 J \psi_{2,4}+ b_5 J \psi_{2,5}+ b_6 J \psi_{2,6}= 0, \label{eq4.2a} \\
&b_2 J \psi_{3,2}+ b_3 J \psi_{3,3}+ b_4 J \psi_{3,4}+ b_5 J \psi_{3,5}+ b_6 J \psi_{3,6}= 0, \label{eq4.2b} \\
&b_2 c_2 K \psi_{2,2}+ b_3 c_3 K \psi_{2,3}+ b_4 c_4 K \psi_{2,4}+ b_5 c_5 K \psi_{2,5}+ b_6 c_6 K \psi_{2,6}=0 \label{eq4.2c},\\
& b_3 J a_{32}J\psi_{2,2}  + b_4 J (a_{42}J\psi_{2,2}+a_{43}J\psi_{2,3})  
+b_5 J(a_{52}J\psi_{2,2}+a_{53}J\psi_{2,3}+a_{54}J\psi_{2,4})  \label{eq4.2d}\\
&+b_6J(a_{62}J\psi_{2,2}+a_{63}J\psi_{2,3}+a_{64}J\psi_{2,4}+ a_{65}J\psi_{2,5})=0. \notag
\end{align}
\end{subequations}
We now solve these order conditions. We note from \eqref{eq3.3b} that 
\begin{equation}\label{eq4.3}
\psi_{2,i} = \sum_{j=2}^{i-1}a_{ij}c_j-c^{2}_i \varphi_{2,i}, \q
\psi_{3,i} = \sum_{j=2}^{i-1}a_{ij}\frac{c^{2}_j}{2!}-c^{3}_i \varphi_{3,i}
\end{equation}
 and thus
$
 \psi_{2,2}=-c^{2}_{2}\varphi_{2,2} \ne 0, \q \psi_{3,2}=-c^{3}_{2}\varphi_{3,2} \ne 0
$
(since  $c_2\ne 0$). Using \eqref{eq4.3}, one can infer that either $ \psi_{2,3}$ or $ \psi_{3,3}$ must be nonzero as well (if both are zero then $a_{32}=\dfrac{c^2_3}{c_2}\varphi_{2,3}=\dfrac{2c^3_3}{c^2_2}\varphi_{3,3}$, which is impossible since $c_3>0$ and $\{\varphi_2, \varphi_3 \}$ are linearly independent). This strongly suggests that  $b_2=b_3=0$  in order to later fulfill  \eqref{eq4.2} in the strong sense with arbitrary square matrices $J$ and $K$. 
Next, we further observe that if $b_4 \ne 0$ one may need both  $ \psi_{2,4}= \psi_{3,4}=0$ (which solves $a_{42} \ne 0$, $a_{43} \ne 0$). However, this makes the second term in \eqref{eq4.2d} to be nonzero which is then very difficult to satisfy   \eqref{eq4.2d} in the strong form. Putting together, it requires that  $b_2=b_3=b_4=0$. Using this sufficient condition we can easily solve \eqref{eq4.1} to get 
\begin{equation*}
 b_5 =\dfrac{-c_6\varphi_2 + 2\varphi_3}{c_5(c_5-c_6)}, \q
 b_6 =\dfrac{-c_5\varphi_2 + 2\varphi_3}{c_6(c_6-c_5)}
\end{equation*}
for any choice of distinct nodes $c_5, c_6>0$, satisfying the condition 
\begin{equation}\label{eq:c5c6}
c_5=\dfrac{4c_6-3}{6c_6-4}.
\end{equation}
Since $b_5, b_6 \ne 0$, we must enforce  $\psi_{2,5}=\psi_{3,5}=0$ and $\psi_{2,6}=\psi_{3,6}=0$ to satisfy conditions \eqref{eq4.2a}--\eqref{eq4.2c}. 
 Using \eqref{eq4.3}, this leads to the following 2 systems of two linear equations
\begin{subequations} \label{eq4.4}
\begin{align}
a_{52} c_2+ a_{53} c_3+ a_{54}c_4&=c^2_5\varphi_{2,5}, \label{eq4.4a} \\
a_{52} c^2_2+ a_{53} c^2_3+ a_{54}c^2_4&=2c^3_5\varphi_{3,5}, \label{eq4.4b}
\end{align}
\end{subequations}
and 
\begin{subequations} \label{eq4.5}
\begin{align}
a_{62} c_2+ a_{63} c_3+ a_{64}c_4+ a_{65}c_5&=c^2_6\varphi_{2,6}, \label{eq4.5a} \\
a_{62} c^2_2+ a_{63} c^2_3+ a_{64}c^2_4+ a_{65}c^2_5&=2c^3_6\varphi_{3,6}. \label{eq4.5b}
\end{align}
\end{subequations}
To satisfy conditions \eqref{eq4.2d}, we further enforce $a_{52}=a_{62}=0$ (since  $\psi_{2,2} \ne 0$), which immediately solves \eqref{eq4.4} for coefficients (with $c_3 \ne c_4$) 
\begin{equation} \label{eq4.6}
a_{53} =\dfrac{-c_4 c^2_5\varphi_{2,5}+ 2 c^3_5\varphi_{3,5}}{c_3(c_3-c_4)} \ne 0,  \q
 a_{54}=\dfrac{-c_3 c^2_5\varphi_{2,5} + 2 c^3_5\varphi_{3,5}}{c_4(c_4-c_3)} \ne 0, 
\end{equation}
and thus we also need  $\psi_{2,3}=\psi_{2,4}=0$ (since $\psi_{2,5}=0$), which gives
 \begin{subequations} \label{eq4.7}
\begin{align}
a_{32} &=\dfrac{c^2_3}{c_2}\varphi_{2,3},  \label{eq4.7a} \\
a_{42} c_2+ a_{43} c_3&=c^2_4\varphi_{2,4}. \label{eq4.7b}
\end{align}
\end{subequations}
After fulfilling all the required order conditions in \eqref{eq4.1}--\eqref{eq4.2}, we see from \eqref{eq4.5} and \eqref{eq4.7b} that either $a_{42}$ or $a_{43}$ and one of the coefficients among $a_{63},\ a_{64},\ a_{65}$  can be taken as free parameters. We now use them to construct parallel stages.
 Guided by \eqref{eq4.6} and \eqref{eq4.7a}, we choose $a_{43}=0$ to make $U_{n4}$ is independent of $U_{n3}$ so that both these stages only depend on $U_{n2}$, and choose $a_{65}=0$ to make $U_{n6}$ is independent of $U_{n5}$ so that both these stages only depend on the two preceding stages $U_{n3}, U_{n4}$ (since $a_{52}=a_{62}=0$). From this we determine the remaining coefficients 
 \begin{equation} \label{eq4.8}
a_{42} =\dfrac{c^2_4}{c_2}\varphi_{2,4}, \q
a_{63} =\dfrac{-c_4 c^2_6\varphi_{2,6}+ 2 c^3_6\varphi_{3,6}}{c_3(c_3-c_4)}, \q
 a_{64}=\dfrac{-c_3 c^2_6\varphi_{2,6}+ 2 c^3_6\varphi_{3,6}}{c_4(c_4-c_3)}. 
\end{equation}
 Putting altogether and rearranging terms in $U_{ni},\ u_{n+1}$ as linear combinations of $\varphi$ functions, we obtain the following family of  4th-order 6-stage methods (with the pairs of parallel stages $\{U_{n3}, U_{n4}\}$ and $\{U_{n5}, U_{n6}\}$),  which will be called $\mathtt{expRK4s6}$:
 \begin{subequations}\label{eq:expRK4s6}
\begin{align}
U_{n2}=u_n &+\varphi_1 (c_2 h A)  c_2 hF(t_n, u_n), \\
U_{n,k}=u_n &+ \varphi_1 (c_k h A) c_k hF(t_n, u_n)+ \varphi_2 (c_k h A) \tfrac{c^2_k}{c_2} h D_{n2},  \q \hspace{1.5cm} k=3, 4 \\
U_{n,j}=u_n &+ \varphi_1 (c_j h A)c_j h F(t_n, u_n)+ \varphi_{2} (c_j hA) \tfrac{c^2_j}{c_3-c_4} h \big(\tfrac{-c_4}{c_3}D_{n3} +\tfrac{c_3}{c_4}D_{n4}\big)\notag\\
&+  \varphi_{3} (c_j hA) \tfrac{2c^3_j}{c_3-c_4} h \big(\tfrac{1}{c_3}D_{n3} -\tfrac{1}{c_4}D_{n4}\big), \q \q \ \hspace{2cm} j=5,6 \\
u_{n+1}=u_n &+ \varphi_1 (h A) h F(t_n, u_n)+ \varphi_{2} (hA) \tfrac{1}{c_5-c_6} h \big(\tfrac{-c_6}{c_5}D_{n5} +\tfrac{c_5}{c_6}D_{n6}\big) \notag \\
&+\varphi_{3} (hA) \tfrac{2}{c_5-c_6} h \big(\tfrac{1}{c_5}D_{n5} -\tfrac{1}{c_6}D_{n6}\big). 
\end{align}
\end{subequations} 
For the numerical experiments given in Section~\ref{sec:num}, we choose $c_2=c_3=\frac{1}{2}, c_4=\frac{1}{3}$, $c_6=\frac{1}{3}$ which gives $c_5=\frac{5}{6}$ due to \eqref{eq:c5c6}.
\begin{remark}\label{remark4.1}
(\emph{A comparison with $\mathtt{expRK4s5}$}).
As seen,  $\mathtt{expRK4s6}$  is resulted from weakening only condition 4 of Table~\ref{tb1} instead of weakening four conditions 4--7 as in the derivation of $\mathtt{expRK4s5}$. 
While the 5-stage method $\mathtt{expRK4s5}$ requires 6 sequential evaluations in each step (as mentioned in Section~2), the new  fourth-order 6-stage method  $\mathtt{expRK4s6}$ requires only 4 sequential evaluations, making it to behave like a 4-stage method. This is due to the fact  $\mathtt{expRK4s6}$  has the pairs of parallel stages $\{U_{n3}, U_{n4}\}$ and $\{U_{n5}, U_{n6}\}$ and all $U_{ni}$ within these pairs have the same format, i.e., same (one) linear combination of $\varphi_k (c_i h A)v_k$, allowing them to be computed in parallel or simultaneously (see Section~\ref{sec:implementation}).
\end{remark}
 
\subsection{A family of  fifth-order methods with parallel stages}
Constructing fifth-order exponential Runge-Kutta methods needs much more effort as one has to solve 16 order conditions in Table~\ref{tb1}. As mentioned in Section~2, the only existing method of order 5 in the literature is $\mathtt{expRK5s8}$ (see \cite{LO14b}) which requires $s=8$ stages. Like $\mathtt{expRK4s5}$, this method does not have any parallel stages and must be implemented in a sequential way. It also does not satisfy the assumption on the order conditions stated in Theorem~\ref{th3.1}. Indeed, it was constructed by fulfilling conditions 1--7 in the strong form and weakening conditions 8--16  (9 out of 16 order conditions) with $b_i(0)$ in place of $b_i (Z)$. This resulted in the last two internal stages $U_{n7}$ and  $U_{n8}$ that involve several different linear combinations of $\varphi_k (c_i hA)v_k$ (with different scalings $c_6, c_7, c_8$ of $hA$), for which additional computational efforts are required to compute those stages (as shown in Section~2).

Therefore,  to derive a method based on Theorem~\ref{th3.1} which later allows us to derive parallel stages schemes with  $U_{ni}$ involving only one linear combination of $\varphi_k (c_i hA)v_k$, we have to increase $s\geqslant 9$. To make it easier for readers to follow, we consider $s=10$ first and later employ the similar procedure to show that it is not possible to fulfill  condition 11 of Table~\ref{tb1} in the strong form (and thus not satisfying Theorem~\ref{th3.1}) with $s=9$.\\
\\
\emph{a) The case $s=10$}: Similarly to the derivation presented in Subsection~\ref{sec4.1}, using \eqref{eq4.3}, it strongly suggests  $b_2=b_3=b_4=b_5=b_6=b_7=0$ in order to solve conditions 3, 5, 9, 7, 16, 13, and 15 in their strong form.  Using this, these conditions now read as 
\begin{subequations} \label{eq4.10}
\begin{align}
&b_8 J \psi_{2,8}+ b_9 J \psi_{2,9}+ b_{10} J \psi_{2,10}= 0, \label{eq4.10a} \\
&b_8 J \psi_{3,8}+ b_9 J \psi_{3,9}+ b_{10} J \psi_{3,10}= 0,  \label{eq4.10b} \\
&b_8 J \psi_{4,8}+ b_9 J \psi_{4,9}+ b_{10} J \psi_{4,10}= 0,  \label{eq4.10c} \\
&b_8 c_8 K \psi_{2,8}+ b_9 c_9 K \psi_{2,9}+ b_{10} c_{10} K \psi_{2,10}=0 \label{eq4.10d},\\
&b_8 c^2_8 L \psi_{2,8}+ b_9 c^2_9 L \psi_{2,9}+ b_{10} c^2_{10} L \psi_{2,10}=0 \label{eq4.10e},\\
&b_8 c_8 K \psi_{3,8}+ b_9 c_9 K \psi_{3,9}+ b_{10} c_{10} K \psi_{3,10}=0 \label{eq4.10f},\\
&b_8 B(\psi_{2,8},\psi_{2,8})+ b_9 B(\psi_{2,9},\psi_{2,9})+ b_{10} B(\psi_{2,10},\psi_{2,10})=0, \label{eq4.10g}
 \end{align}
\end{subequations}
respectively. And conditions 1, 2, 4, and 8 (weakened form) become
 \begin{subequations} \label{eq4.9}
\begin{align}
 b_8 c_8+b_9 c_9+b_{10} c_{10} &=\varphi_2, \\
 b_8 c^2_8+b_9 c^2_9+b_{10} c^2_{10}&=2\varphi_3,\\
b_8 c^3_8+b_9 c^3_9+b_{10} c^3_{10}&=6\varphi_4,\\
b_8 (0) c^4_8+b_9 (0) c^4_9+b_{10}(0) c^4_{10}&=24\varphi_5(0)=1/5.
\end{align}
\end{subequations}
 Solving \eqref{eq4.9} gives 
 \begin{subequations} \label{eq4.11}
\begin{align} 
 b_8& =\dfrac{c_9 c_{10}\varphi_2 - 2(c_9+c_{10})\varphi_3+6\varphi_4}{c_8(c_8-c_9)(c_8-c_{10})},\\
 b_9& =\dfrac{c_8 c_{10}\varphi_2 - 2(c_8+c_{10})\varphi_3+6\varphi_4}{c_9(c_9-c_8)(c_9-c_{10})},\\ 
  b_{10}& =\dfrac{c_8 c_9\varphi_2 - 2(c_8+c_9)\varphi_3+6\varphi_4}{c_{10}(c_{10}-c_8)(c_{10}-c_9)}
\end{align}
\end{subequations}
where   $c_8, c_9$, and $c_{10}$ are distinct and positive nodes satisfying the algebraic equation
\begin{equation}\label{eq4.12}
\dfrac{c_8+c_9+c_{10}}{4}-\dfrac{c_8 c_9 + c_8 c_{10} + c_9 c_{10}}{3}+\dfrac{c_8 c_9 c_{10}}{2}=\dfrac{1}{5}.
\end{equation}
Clearly, $b_8, b_9, b_{10} \ne 0$ so one has to enforce
\begin{equation}\label{eq4.13} 
\psi_{2,j}=\psi_{3,j}=\psi_{4,j}=0  \ (j=8, 9,10)
\end{equation}
to satisfy \eqref{eq4.10}  in the strong sense with arbitrary square matrices $J, K, L$ and $B$. 
Next, we consider conditions 6 and 10 taken into account that $b_i=0 $ ($i=2,\cdots,7$) and \eqref{eq4.13}, which can be now simplified as 
\begin{equation}\label{eq4.14} 
\sum_{j=2}^{7} (b_8 J a_{8j}+ b_9 J a_{9j}+ b_{10} J a_{10j}) J \psi_{m,j}=0    \ \  (m=2, 3),
\end{equation}
respectively.
In order to satisfy the strong form of   \eqref{eq4.14}  one needs
 \begin{equation}\label{eq4.15} 
a_{8j}=a_{9j}=a_{10j}=0  \ (j=2, 3, 4)
\end{equation}
(this is again due to \eqref{eq4.3}) and 
\begin{equation}\label{eq4.16} 
\psi_{2,j}=\psi_{3,j}=0  \ (j=5, 6,7).
\end{equation}
With \eqref{eq4.15}, we note that $U_{n8}, U_{n9}, U_{n10}$ are independent of the internal stages $U_{n2}, U_{n3}, U_{n4}$.
Taking into all the requirements above, one can easily see that conditions 12 and 14 are now automatically fulfilled. Therefore, the only remaining condition to satisfy is condition 11. 

Before working with condition 11, we first solve   \eqref{eq4.13} using  \eqref{eq4.15}. 
For this, we observe that several coefficients $a_{ij}$ can be considered as free parameters. To have $U_{n8}, U_{n9}, U_{n10}$ are independent of each other, we choose 
 \begin{equation}\label{eq4.16add} 
a_{98}=a_{10,8}=a_{10,9}=0. 
\end{equation}
The resulting systems of linear equations from  \eqref{eq4.13} is then solved with the unique solution
\begin{equation}\label{eq4.17} 
a_{ij} =\dfrac{c^2_i c_k c_l \varphi_{2,i} -2 c^3_i (c_k +c_l)\varphi_{3,i}+ 6 c^4_i \varphi_{4,i}}{c_j(c_j-c_k)(c_j - c_l)}, \q
i=8, 9,10; \ j, k, l \in \{5, 6, 7\}, \ j \ne k \ne l 
\end{equation}
(i.e., $c_5, c_6, c_7>0$ are distinct nodes).\\
We now use $b_i=0 $ ($i=2,\cdots,7$), \eqref{eq4.13}, \eqref{eq4.15},  \eqref{eq4.16}, and \eqref{eq4.16add} to simplify condition 11 as 
\begin{equation}\label{eq4.18} 
\sum_{i=8}^{10} b_i J \sum_{j=5}^{7} a_{ij} J \big(a_{j2}J\psi_{2,2}+ a_{j3}J\psi_{2,3}+a_{j4}J\psi_{2,4}\big)=0.
\end{equation}
Since $b_8, b_9, b_{10} \ne 0$, coefficients  $a_{ij}$ in  \eqref{eq4.17} ($i \in \{8, 9, 10\}, j \in \{5, 6, 7\}$) are also nonzero, and that $\psi_{2,2} \ne 0$, we must enforce 
 \begin{equation}\label{eq4.19} 
a_{j2}=0  \ (j=5, 6, 7), \ \text{i.e.,} \ a_{52}=a_{62}=a_{72}=0
\end{equation}
and require that
\begin{equation}\label{eq4.20} 
\psi_{2,3}=\psi_{2,4}=0
\end{equation}
in order to satisfy  \eqref{eq4.18} in the strong sense.
Note, because of  \eqref{eq4.19},  one could not require $a_{53}=0$ or $a_{54}=0$ ($j=5$) in \eqref{eq4.18} or both as this does not agree with the requirement $\psi_{2,5}=\psi_{3,5}=0$ in \eqref{eq4.16}  (in other words, the linear system of equations displayed in \eqref{eq4.4} represented for this requirement has no solution). This justifies the requirement  \eqref{eq4.20}.

Finally, we solve \eqref{eq4.20} and  \eqref{eq4.16} for the remaining coefficients $a_{ij}$. 
When solving  \eqref{eq4.20}  (see \eqref{eq4.7}), we choose $a_{43}=0$ to have $U_{n4}$ is independent of $U_{n3}$. This gives 
\begin{equation}\label{eq4.21} 
a_{32} =\dfrac{c^2_3}{c_2}\varphi_{2,3}, \q  a_{42} =\dfrac{c^2_4}{c_2}\varphi_{2,4}.
\end{equation}
When solving \eqref{eq4.16} (using \eqref{eq4.19}), we choose $a_{65}=a_{75}=a_{76}=0$ to have $U_{n5}, U_{n6}, U_{n7}$ are independent of each other. This results in the following 6 coefficients:
\begin{equation}\label{eq4.22} 
a_{ij} =\dfrac{ -c^2_i  c_k \varphi_{2,i} + 2 c^3_i \varphi_{3,i}}{c_j (c_j-c_k)}, \q
i=5, 6,7; \ j, k  \in \{3, 4\}, \ j \ne k  
\end{equation}
(i.e., $c_3, c_4 >0$ are distinct nodes).

Inserting all the obtained coefficients $a_{ij}$ and $b_i$ into $U_{ni},\ u_{n+1}$ and rewriting these stages as linear combinations of $\varphi$ functions, we obtain the following family of  5th-order 10-stage methods  (with the groups of parallel stages $\{U_{n3}, U_{n4}\}$,  $\{U_{n5}, U_{n6}, U_{n7}\}$, and  $\{U_{n8}, U_{n9}, U_{n10}\}$) which will be called $\mathtt{expRK5s10}$:
\begin{equation*}\label{eq:expRK5s10}
\begin{aligned}
    U_{n2} = u_n &+ \varphi_1 (c_2 hA) c_2 hF(t_n,u_n), \\
    U_{n\ell} = u_n &+ \varphi_1 (c_{\ell} hA) c_{\ell} hF(t_n, u_n) + \varphi_2 (c_{\ell} hA) \tfrac{c^2_{\ell}}{c_2} h D_{n2},  \hspace{3.3cm} \ell=3,4 \\
 U_{nm} = u_n &+ \varphi_1 (c_m hA)c_m hF(t_n, u_n)+ \varphi_{2} (c_m hA) c^2_m  h \big(\tfrac{c_4 }{c_3 (c_4-c_3)}D_{n3}  +\tfrac{c_3 }{c_4 (c_3-c_4)} D_{n4}\big)\\
    &+  \varphi_{3} (c_m hA) c^3_m h \big(\tfrac{2}{c_3 (c_3-c_4)} D_{n3} -\tfrac{2}{c_4 (c_3-c_4)} D_{n4}\big), \q \hspace{2.5cm} m=5,6,7 \\
    U_{nq} = u_n &+ \varphi_1 (c_q hA)c_q hF(t_n, u_n)
    + \varphi_{2} (c_q hA) c^2_q  h \big(\alpha_5 D_{n5} +\alpha_6 D_{n6}+\alpha_7 D_{n7} \big)\\
    &+ \varphi_{3} (c_q hA) c^3_q h \big(\beta_5 D_{n5} -\beta_6 D_{n6}-\beta_7 D_{n7}\big) \\
    &+ \varphi_{4} (c_q hA) c^4_q h \big(\gamma_5 D_{n5} +\gamma_6 D_{n6}+\gamma_7 D_{n7}\big), \hspace{3.4cm} q=8,9,10 \\
    u_{n+1} = u_n &+ \varphi_1 (hA) h F(t_n, u_n)+ \varphi_{2} (hA) h \big(\alpha_8 D_{n8} + \alpha_9 D_{n9} +\alpha_{10} D_{n10} \big) \\
    &-\varphi_{3} (hA)  h \big(\beta_8 D_{n8} + \beta_9 D_{n9} +\beta_{10} D_{n10} \big)
    +\varphi_{4} (hA)  h \big(\gamma_8 D_{n8} + \gamma_9 D_{n9} +\gamma_{10} D_{n10} \big),
  \end{aligned}
\end{equation*} 
where
\begin{equation}\label{eq:coefficients}
   \alpha_i = \dfrac{c_k c_l}{c_i (c_i-c_k)(c_i - c_l)},\q  
    \beta_i = \dfrac{2(c_k+ c_l)}{c_i (c_i-c_k)(c_i - c_l)}, \q 
\gamma_i = \dfrac{6}{c_i (c_i-c_k)(c_i - c_l)}
\end{equation} 
with  $i \in \{5, 6, 7\}$ for $\ k, l \in \{5, 6, 7\}$, and $i \in \{8, 9, 10\}$ for $\ k, l \in \{8, 9, 10\}$ (note that $i, k, l$ are distinct indices and that $c_i, c_k, c_l$ are distinct (positive) nodes). \\
For our numerical experiments, we choose
$c_2=c_3=c_5=\tfrac{1}{2}$, $c_4=c_6=\tfrac{1}{3}$, $c_7=\tfrac{1}{4}$,
$c_8=\tfrac{3}{10}$, $c_9=\tfrac{3}{4}$, and $c_{10}=1$ (satisfying  \eqref{eq4.12}). 
\begin{remark}\label{remark4.2}
(\emph{A comparison with  $\mathtt{expRK5s8}$}).
Although the new fifth-order method $\mathtt{expRK5s10}$ has 10 stages (compared to 8 stages of $\mathtt{expRK5s8}$ displayed in Section~2), its special structure offers much more efficient for implementation. In particular, all $U_{ni}$ in this scheme involve only one linear combination of $\varphi_k (c_i h A)v_k$ which can be computed by one evaluation for each; and more importantly, due to the same format of multiple stages within each of the three groups $\{U_{n3}, U_{n4}\}$,  $\{U_{n5}, U_{n6}, U_{n7}\}$, and  $\{U_{n8}, U_{n9}, U_{n10}\}$ (same linear combination with different inputs $c_i$), they can be computed simultaneously or implemented in parallel (see Section~\ref{sec:implementation}). This makes $\mathtt{expRK5s10}$ to behave like a 5-stage method only, thereby requiring only 5 sequential evaluations in each step.
Moreover, while $\mathtt{expRK5s8}$ requires weakening 9 out of 16 order conditions of Table~\ref{tb1},  $\mathtt{expRK5s10}$ requires only one condition (number 8) held in the weakened form. Note that by following the similar way of deriving  $\mathtt{expRK5s10}$, we can derive a scheme that satisfies all the stiff order conditions  in Table~\ref{tb1} in the strong sense with $s=11$. Such a scheme, however, still behaves like a 5-stage method. Therefore, we do not  discuss further on this case.
\end{remark}
\emph{b) The case $s=9$ (which does not work)}: Clearly, in this case we have less degree of freedoms than the case $s=10$ when solving the order conditions in  Table~\ref{tb1}. Nonetheless, one can still proceed in a similar way as done for $s=10$. Again, it strongly suggests  $b_2=b_3=b_4=b_5=b_6=0$ (which solves for $b_7, b_8, b_{9} \ne 0$ from conditions 1, 2, 4) and 
\begin{equation}\label{eq4.27} 
\psi_{2,j}=\psi_{3,j}=\psi_{4,j}=0  \ (j=7, 8, 9)
\end{equation}
in order to satisfy conditions 1, 2,  3, 4, 5, 7, 9, 13, 15, 16  in the strong form. 
With this, conditions 6 and 10 now become 
\begin{equation}\label{eq4.28} 
\sum_{j=2}^{6} (b_7 J a_{7j}+ b_8 J a_{8j}+ b_{9} J a_{9j}) J \psi_{m,j}=0   \ \  (m=2, 3).
\end{equation}
Again, due to the fact that $\psi_{2,2}, \psi_{3,2} \ne 0$ and either $ \psi_{2,3}$ or $ \psi_{3,3}$ must be nonzero, one needs to enforce  
$
a_{7j}=a_{8j}=a_{9j}=0  \ (j=2, 3)
$
in  \eqref{eq4.28}.
Using this to solve \eqref{eq4.27}  for $j=7$ ($\psi_{2,7}=\psi_{3,7}=\psi_{4,7}=0$) gives a unique solution (with $c_4, c_5, c_6 >0$ and are distinct) for  $a_{74}, a_{75}, a_{76} \ne 0$, which then determines $U_{n7}$. Next, one can  solve \eqref{eq4.27}  for  $j=8, 9$  to obtain $ U_{n8}, U_{n9}$ that are independent of $U_{n7}$, as well as  are independent of each other, by requiring the three free parameters $a_{87}=a_{97}=a_{98}=0$. As a result, one gets $a_{7j}, a_{8j}, a_{9j} \ne 0  \ (j=5, 6)$. This immediately suggests 
$
\psi_{2,j}=\psi_{3,j}=0  \ (j=4, 5, 6)
$
to completely fulfill  \eqref{eq4.28}  with arbitrary square matrix $J$. 
With all of these in place, conditions 12 and 14 are automatically fulfilled, and condition 11 is now reduced to 
\begin{equation}\label{eq4.29} 
\sum_{i=7}^{9} b_i J \sum_{j=4}^{6} a_{ij} J \big(a_{j2}J\psi_{2,2}+ a_{j3}J\psi_{2,3}\big)=0.
\end{equation}
Clearly, since $b_7, b_8, b_{9} \ne 0$, $a_{7j}, a_{8j}, a_{9j} \ne 0  \ (j=4, 5, 6)$, and $\psi_{2,2} \ne 0$,   \eqref{eq4.29} can be satisfied in the strong sense only if we have one of the following conditions: $a_{j2}=a_{j3}=0 $ or $a_{j2}=\psi_{2,3}=0, \ (j=4, 5, 6)$.
Unfortunately,  either of these requirements is in contradiction with $
\psi_{2,j}=\psi_{3,j}=0  \ (j=4, 5, 6)$ which is needed for conditions 6 and 10 mentioned above. For example, solving $\psi_{2,4}=\psi_{3,4}=0$ results in $a_{42}, a_{43} \ne 0$.

\section{Details implementation of fourth- and fifth-order schemes}\label{sec:implementation}
In this section, we present details implementation of the old and new fourth- and fifth-order expRK schemes  ($\mathtt{expRK4s5}$, $\mathtt{expRK5s8}$,  $\mathtt{expRK4s6}$, $\mathtt{expRK5s10}$) mentioned above.

As mentioned in Section~\ref{sec2.1}, we will use the MATLAB routine \texttt{phipm\_simul\_iom} (described in details in \cite{Luan18}) 
 to implement expRK methods. 
 In particular, given the following inputs: an array of scaling factors $\mathtt{t}=[\rho_1, \cdots, \rho_{r}]$ with $0<\rho_1<\rho_2<\cdots<\rho_r \le 1$ ($\mathtt{t}$ could be a positive scalar), 
an $n$-by-$n$ matrix $M$, and a set of  column vectors  $\mathtt{V}=[V_0,\ldots,v_q]$ (each $v_i$ is an $n$-by-$1$ vector), a tolerance $\mathtt{tol}$, an initial value $m$ for the dimension of the Krylov subspace, and an incomplete orthogonalization length of  $\mathtt{iom}$, a call to this function
\begin{equation} \label{eq5.1}
\mathtt{phipm\_simul\_iom(t, M, V, tol, m, iom)}
\end{equation}
simultaneously computes the following $r$ linear combinations 
\begin{equation} \label{eq5.2}
L_{\rho_i ,\mathtt{V}}= \varphi_0 (\rho_i M)v_0 + \varphi_1 (\rho_i M) \rho_i v_1 +\varphi_2 (\rho_i M) \rho_{i}^2 v_2+\cdots+\varphi_{q}(\rho_i  M) \rho_{i} ^q v_{q},  \ 1 \le i \le r.
\end{equation}
Note that, by setting $V_{j}=\rho_{i}^{j} v_j $ ($j=0,\cdots,q$), \eqref{eq5.2} becomes  \eqref{eq2.2}.  In other words, all the linear combinations in \eqref{eq2.2} ( if $V_j$ are given instead of $v_j$) can be then computed at the same time with one call \eqref{eq5.1}  by using scaled vectors $v_j= V_{j}/ \rho_{i}^{j}$ for the input $\mathtt{V}$.

In the following, we set 
\begin{equation} \label{eq5.3}
 M=hA, \q \rho_i= c_i, \q  v=h F (t_n, u_n), \q  d_{i}=hD_{ni}. 
\end{equation}
to simplify notations in presenting details of implementation of the fourth- and fifth-order methods mentioned above. When calling \eqref{eq5.1}, we use $\mathtt{tol}=10^{-12}$,  $\mathtt{m}=1$ (default value), and $\mathtt{imo}=2$ (as in  \cite{Luan18}).

\noindent \emph{Implementation of $\mathtt{expRK4s5}$} 
($c_2=c_3=c_5=\tfrac{1}{2}, c_4=1 $):  
As discussed in Remark~\ref{remark2.1}, $\mathtt{expRK4s5}$ requires  a sequential implementation of the following 6 different linear combinations  of the form \eqref{eq5.2}, corresponding to 6 calls to \texttt{phipm\_simul\_iom}:
\begin{itemize}
\item[(i)] Evaluate  $L_{c_2,\mathtt{V}}$ with $\mathtt{t}=c_2, \mathtt{V}=[0, v]$ to get $U_{n2}=u_n + L_{c_2,\mathtt{V}}$.

\item[(ii)] Evaluate  $L_{c_3,\mathtt{V}}$ with  $\mathtt{t}=c_3, \mathtt{V}=[0, v, d_2/c^2_3]$ to get $U_{n3}=u_n + L_{c_3,\mathtt{V}}$.

\item[(iii)] Evaluate  $L_{c_4,\mathtt{V}}$  with  $\mathtt{t}=c_4, \mathtt{V}=[0, v, d_2 +d_3]$ to get $U_{n4}=u_n + L_{c_4,\mathtt{V}}$.

\item[(iv)] Evaluate  $L_{c_5,\mathtt{V_1}}$  with  $\mathtt{t}=c_5, \mathtt{V}_1=[0, v, 2 d_{2} + 2 d_{3} - d_{4},(-d_{2} - d_{3} + d_{4})/c^2_5] $ 
and
\item[(v)] Evaluate  $L_{c_4,\mathtt{V_2}}$  with  $\mathtt{t}=c_4, \mathtt{V}_2=[0, 0, ( d_{2} +  d_{3} - d_{4})/4,(-d_{2} - d_{3} + d_{4})] $ \\
to get $U_{n5}=u_n + L_{c_5,\mathtt{V}_1}+L_{c_4,\mathtt{V_2}} $.

\item[(vi)] Evaluate  $L_{1,\mathtt{V}}$ with $\mathtt{t}=1, \mathtt{V}=[0,v, -d_4 +5d_5, 4d_4-8d_5 ]$ to get $u_{n+1}=u_n + L_{1,\mathtt{V}}$.
\end{itemize}
Since $d_i=hD_{ni}$ which depends on $U_{ni}$, these are the  6 (sequential) evaluations.\\

\noindent \emph{Implementation of $\mathtt{expRK4s6}$}
($c_2=c_3=\frac{1}{2}, c_4=c_6=\frac{1}{3}, c_5=\frac{5}{6}$):  
As discussed in Remark~\ref{remark4.1},  $\mathtt{expRK4s6}$ can be implemented like a 4-stage method by evaluating the following 4 sequential evaluations, corresponding to 4 calls to \texttt{phipm\_simul\_iom}:
\begin{itemize}
\item[(i)] Evaluate  $L_{c_2,\mathtt{V}}$ with $\mathtt{t}=c_2, \mathtt{V}=[0, v]$ to get $U_{n2}=u_n + L_{c_2,\mathtt{V}}$.

\item[(ii)] Evaluate  $L_{c_4,\mathtt{V}}$ and $L_{c_3,\mathtt{V}}$ simultaneously using  $\mathtt{t}=[c_4, c_3], \mathtt{V}=[0, v, d_2/c_2]$ to get both $U_{n3}=u_n + L_{c_3,\mathtt{V}}$ and $U_{n4}=u_n + L_{c_4,\mathtt{V}}$.

\item[(iii)] Evaluate  $L_{c_5,\mathtt{V}}$ and $L_{c_6,\mathtt{V}}$ simultaneously  with $\mathtt{t}=[c_6, c_5],\\ \mathtt{V}=[0, v, \tfrac{-c_4}{(c_3-c_4)c_3}d_3 +\tfrac{c_3}{(c_3-c_4)c_4}d_4,  \tfrac{1}{(c_3-c_4)c_3}d_3 -\tfrac{1}{(c_3-c_4)c_4}d_4]$ 
to get both $U_{n5}=u_n + L_{c_5,\mathtt{V}}$ and $U_{n6}=u_n + L_{c_6,\mathtt{V}}$.

\item[(iv)] Evaluate  $L_{1,\mathtt{V}}$ with $\mathtt{t}=1, \mathtt{V}=[0,v, \tfrac{1}{c_5-c_6}(\tfrac{-c_6}{c_5}d_5 + \tfrac{c_5}{c_6}d_6), \tfrac{2}{c_5-c_6}(\tfrac{1}{c_5}d_5 - \tfrac{1}{c_6}d_6) ]$ to get $u_{n+1}=u_n + L_{1,\mathtt{V}}$.
\end{itemize}

\noindent \emph{Implementation of $\mathtt{expRK5s8}$}
($c_2=c_3=c_5=\frac{1}{2}, c_4=\frac{1}{4}, c_6=\frac{1}{5}, c_7=\frac{2}{3}, c_8=1$):  
As discussed in Remark~\ref{remark2.1},   $\mathtt{expRK5s8}$ requires  a sequential implementation of 11 different linear combinations  of the form \eqref{eq5.2}, corresponding to the following
 11 calls to \texttt{phipm\_simul\_iom}:
\begin{itemize}
\item[(i)] 
Evaluate  $L_{c_2,\mathtt{V}}$ with $\mathtt{t}=c_2, \mathtt{V}=[0, v]$ to get $U_{n2}=u_n + L_{c_2,\mathtt{V}}$.

\item[(ii)] 
Evaluate  $L_{c_3,\mathtt{V}}$ with  $\mathtt{t}=c_3, \mathtt{V}=[0, v, d_2/c^2_3]$ to get $U_{n3}=u_n + L_{c_3,\mathtt{V}}$.

\item[(iii)] 
Evaluate  $L_{c_4,\mathtt{V}}$ with  $\mathtt{t}=c_4, \mathtt{V}=[0, v, d_3/c^2_4]$ to get $U_{n4}=u_n + L_{c_4,\mathtt{V}}$.

\item[(iv)] 
Evaluate  $L_{c_5,\mathtt{V}}$ with  $\mathtt{t}=c_5, \mathtt{V}=[0, v, (- d_{3} + 4d_{4})/c^2_5, (2d_{3} - 4d_{4})/c^3_5 ]$ to get $U_{n5}=u_n + L_{c_5,\mathtt{V}}$.

\item[(v)] 
Evaluate  $L_{c_6,\mathtt{V}}$ with  $\mathtt{t}=c_6, \mathtt{V}=[0, v, (8d_{4} - 2d_{5})/c^2_6, (-32d_{4} +16d_{5})/c^3_6 ]$ to get $U_{n6}=u_n + L_{c_6,\mathtt{V}}$.

\item[(vi)] 
Evaluate  $L_{c_7,\mathtt{V_1}}$  with  $\mathtt{t}=c_7, \mathtt{V}_1=[0, v, (\tfrac{-16}{27} d_{5} +\tfrac{100}{27}  d_{6})/c^2_7, (\tfrac{320}{81} d_{5} -\tfrac{800}{81}  d_{n6})/c^3_7] $ 
and
\item[(vii)] 
Evaluate  $L_{c_6,\mathtt{V_2}}$  with  $\mathtt{t}=c_6, 
\mathtt{V}_2=[0, 0, (\tfrac{-20}{81} d_{4} +\tfrac{5}{243}  d_{5}+\tfrac{125}{486}  d_{6})/c^2_6, (\tfrac{16}{81} d_{4} -\tfrac{4}{243}  d_{5}-\tfrac{50}{243}  d_{6})/c^3_6] $ 
to get $U_{n7}=u_n + L_{c_7,\mathtt{V}_1}+L_{c_6,\mathtt{V_2}} $.

\item[(viii)] 
Evaluate  $L_{c_8,\mathtt{V_1}}$  with  $\mathtt{t}=c_8, \mathtt{V}_1=[0, v, (\tfrac{-16}{3} d_{5} +\tfrac{250}{21}  d_{6}+\tfrac{27}{14}  d_{7})/c^2_8, (\tfrac{208}{3} d_{5} -\tfrac{250}{3}  d_{6}- 27 d_{7})/c^3_8, (-240d_{5} + \tfrac{1500}{7}  d_{6}+  \tfrac{810}{7} d_{7})/c^4_8 ] $ 
and
\item[(ix)] 
Evaluate  $L_{c_6,\mathtt{V_2}}$  with  $\mathtt{t}=c_6, \\
\mathtt{V}_2=[0, 0, (\tfrac{-4}{7} d_{5} + \tfrac{25}{49}  d_{6} +\tfrac{27}{98} d_{7}) /c^2_6, (\tfrac{8}{5} d_{5} - \tfrac{10}{7}  d_{6} - \tfrac{27}{35} d_{7})/c^3_6, (\tfrac{-48}{35} d_{5} + \tfrac{60}{49}  d_{6} + \tfrac{162}{245} d_{7})/c^4_6]$ 
and
\item[(x)] 
Evaluate  $L_{c_7,\mathtt{V_3}}$  with  $\mathtt{t}=c_7, 
\mathtt{V}_3=[0, 0, (\tfrac{-288}{35} d_{5} + \tfrac{360}{49}  d_{6} +\tfrac{972}{245} d_{7}) /c^2_7, (\tfrac{384}{5} d_{5} - \tfrac{480}{7}  d_{6} - \tfrac{1296}{35} d_{7})/c^3_7, (\tfrac{-1536}{7} d_{5} + \tfrac{9600}{49}  d_{6} + \tfrac{5184}{49} d_{7})/c^4_7] $ \\
to get $U_{n8}=u_n + L_{c_8,\mathtt{V}_1}+ L_{c_6,\mathtt{V}_2}+L_{c_7,\mathtt{V_3}} $.

\item[(xi)] 
Evaluate  $L_{1,\mathtt{V}}$ with $\mathtt{t}=1, \mathtt{V}=[0,v, \tfrac{125}{14} d_{6} -\tfrac{27}{14} d_{7} +  \tfrac{1}{2} d_{8}, \tfrac{-625}{14} d_{6} + \tfrac{162}{7} d_{7} -  \tfrac{13}{2} d_{8},\tfrac{1125}{14} d_{6} -  \tfrac{405}{7} d_{7}  +  \tfrac{45}{2} d_{8} ]$ to get $u_{n+1}=u_n + L_{1,\mathtt{V}}$.
\end{itemize}

\noindent \emph{Implementation of $\mathtt{expRK5s10}$}
($c_2=c_3=c_5=\tfrac{1}{2}$, $c_4=c_6=\tfrac{1}{3}$, $c_7=\tfrac{1}{4}$,
$c_8=\tfrac{3}{10}$, $c_9=\tfrac{3}{4}$, and $c_{10}=1$): 
As discussed in Remark~\ref{remark4.2},  $\mathtt{expRK5s10}$ can be implemented like a 5-stage method by evaluating the following 5 sequential evaluations, corresponding to 5 calls to \texttt{phipm\_simul\_iom}:
\begin{itemize}
\item[(i)] 
Evaluate  $L_{c_2,\mathtt{V}}$ with $\mathtt{t}=c_2, \mathtt{V}=[0, v]$ to get $U_{n2}=u_n + L_{c_2,\mathtt{V}}$.

\item[(ii)] 
Evaluate  $L_{c_4,\mathtt{V}}$ and $L_{c_3,\mathtt{V}}$ simultaneously using  $\mathtt{t}=[c_4, c_3], \mathtt{V}=[0, v, d_2/c_2]$ to get both $U_{n3}=u_n + L_{c_3,\mathtt{V}}$ and $U_{n4}=u_n + L_{c_4,\mathtt{V}}$.

\item[(iii)] 
Evaluate  $L_{c_5,\mathtt{V}}$,  $L_{c_6,\mathtt{V}}$, and $L_{c_7,\mathtt{V}}$ simultaneously  using 
$\mathtt{t}=[c_7, c_6, c_5]$,\\ 
$\mathtt{V}=[0, v, \tfrac{c_4 }{c_3 (c_4-c_3)}d_{3}  +\tfrac{c_3 }{c_4 (c_3-c_4)} d_{4},  \tfrac{2}{c_3 (c_3-c_4)} d_{3} -\tfrac{2}{c_4 (c_3-c_4)} d_{4}]$ \\
to get  $U_{n5}=u_n + L_{c_5,\mathtt{V}}$, $U_{n6}=u_n + L_{c_6,\mathtt{V}}$, $U_{n7}=u_n + L_{c_7,\mathtt{V}}$.

\item[(iv)] 
Evaluate  $L_{c_8,\mathtt{V}}$,  $L_{c_9,\mathtt{V}}$, and $L_{c_{10},\mathtt{V}}$ simultaneously  using 
$\mathtt{t}=[c_9, c_{10}, c_8]$,\\ 
$\mathtt{V}=[0, v, \alpha_5 d_{5} +\alpha_6 d_{6}+\alpha_7 d_{7}, \beta_5 d_{5} -\beta_6 d_{6}-\beta_7 d_{7}, \gamma_5 d_{5} +\gamma_6 d_{6}+\gamma_7 d_{7} ]$ \\
to get  $U_{n8}=u_n + L_{c_8,\mathtt{V}}$, $U_{n9}=u_n + L_{c_9,\mathtt{V}}$, $U_{n10}=u_n + L_{c_{10},\mathtt{V}}$.

\item[(v)] 
Evaluate  $L_{1,\mathtt{V}}$ with 
$\mathtt{t}=1, \mathtt{V}=[0,v, \alpha_8 d_{8} + \alpha_9 d_{9} +\alpha_{10} d_{10}, 
\beta_8 d_{8} + \beta_9 d_{9} +\beta_{10} d_{10},
 \gamma_8 d_{8} + \gamma_9 d_{9} +\gamma_{10} d_{10} ]$ 
to get $u_{n+1}=u_n + L_{1,\mathtt{V}}$
(coefficients $\alpha_i, \beta_i, \gamma_i$ are given in \eqref{eq:coefficients}).  
\end{itemize}

\section{Numerical experiments}\label{sec:num}
In this section, 
we demonstrate the efficiency of our newly derived fourth- and fifth-order expRK time integration methods ($\mathtt{expRK4s6}$, $\mathtt{expRK5s10}$). Specifically, we will compare their performance against the existing methods of the same orders ($\mathtt{expRK4s5}$, $\mathtt{expRK5s8}$) on several examples of stiff PDEs.
All the numerical simulations  are performed in MATLAB on a single workstation, using an iMac 3.6 GHz Intel Core i7, 32 GB 2400 MHz DDR4.

          
\begin{example}\label{ex1}\rm
(\emph{A one-dimensional semilinear parabolic problem} \cite{HO05}): 
We first verify the order of convergence for the new derived fourth- and fifth-order expRK schemes ($\mathtt{expRK4s6}$, $\mathtt{expRK5s10}$) by considering the following PDE for $u(x,t)$,  $x\in [0,1], t\in [0,1]$, and subject to homogeneous Dirichlet boundary conditions,
\begin{equation} \label{example1}
\frac{\partial u(x,t)}{\partial t}- \frac{\partial^2 u(x,t)}{\partial x^2}   =\frac{1}{1+u^{2}(x,t)}+\Phi (x,t),
\end{equation}
whose exact solution is known to be $u(x,t)=x(1-x)\ee^t$ for a suitable choice of the source function $\Phi (x,t)$. 
\end{example}
\emph{Spatial discretization}:  For this example, we use standard second order finite differences with $200$ grid points. This leads to a very stiff system of the form  \eqref{eq1.1} (with $\|A\|_{\infty} \approx1.6 \times10^5$). 

The resulting system is then integrated on the time interval $[0, 1]$ 
using constant step sizes, corresponding to the number of time steps  $N=4, 8, 16, 32, 64$. 
The time integration errors at the final time $t=1$ are measured  in the maximum norm.

 
 In Figure~\ref{fig6.1}, we plot orders for all the employed integrators in the left diagram and the total CPU time versus the global errors in the right diagram. The left  diagram clearly shows a perfect agreement with our convergence result  in Theorem~\ref{th3.1}, meaning that the two new integrators $\mathtt{expRK4s6}$ and $\mathtt{expRK5s10}$ fully achieve orders 4 and 5, respectively. When compared to the old integrators of  the same orders $\mathtt{expRK4s5}$ and $\mathtt{expRK5s8}$, we note that, given the same number of time steps, $\mathtt{expRK4s6}$ is slightly more accurate but is much faster than $\mathtt{expRK4s5}$ (see the right diagram). In a similar manner,  $\mathtt{expRK5s10}$ gives almost identical global errors but is also much faster than  $\mathtt{expRK5s8}$. Finally, we observe that, for this example,  for a global error that is larger than $10^{-6}$, the new fourth-order method $\mathtt{expRK4s6}$ is the fastest one, and for more stringent errors,  $\mathtt{expRK5s10}$ is the fastest integrator.
\begin{figure}[H]
\centering
\begin{tabular}{cc}
\epsfig{file=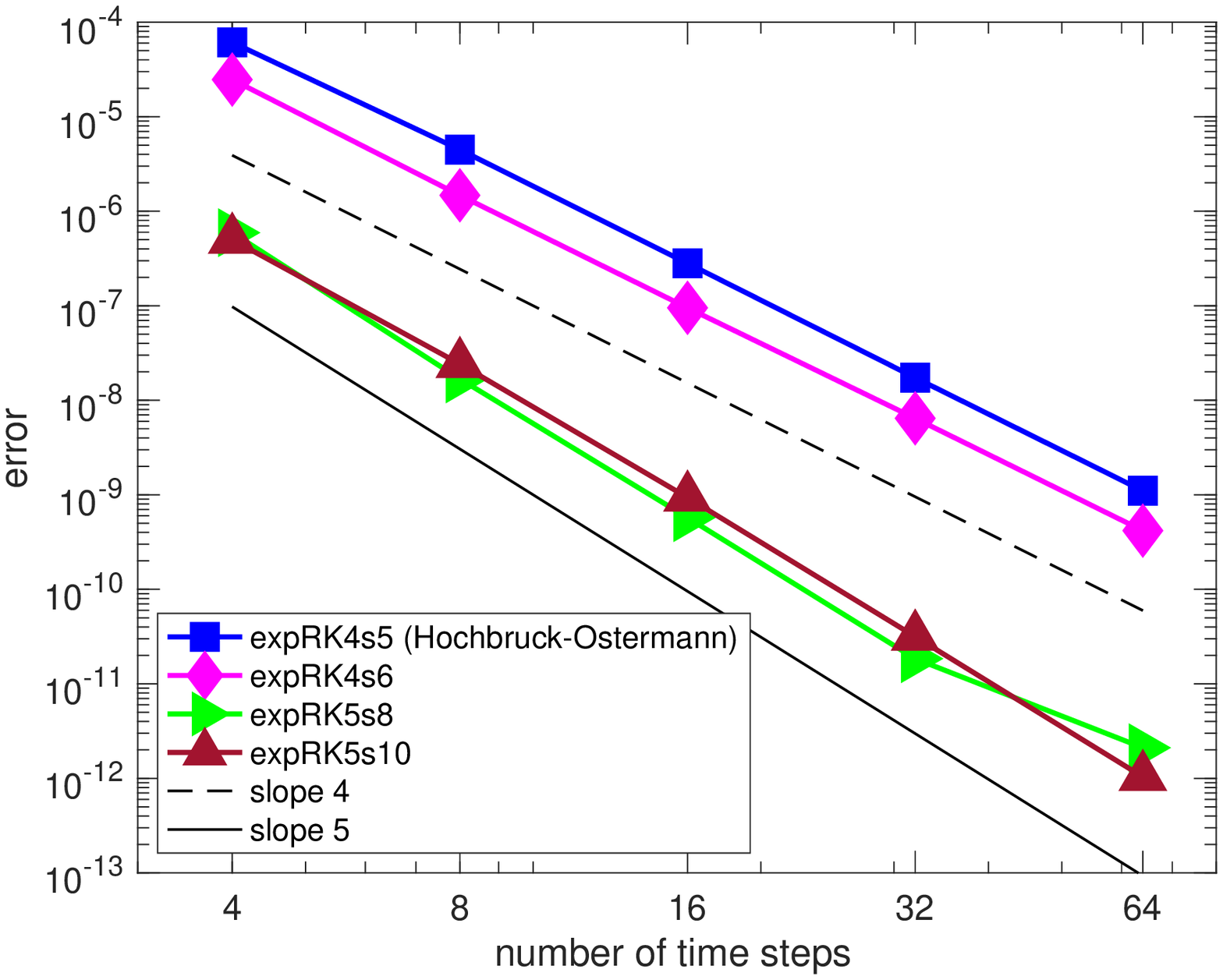,width=0.47\linewidth,clip=}
\hspace{2mm}
\epsfig{file=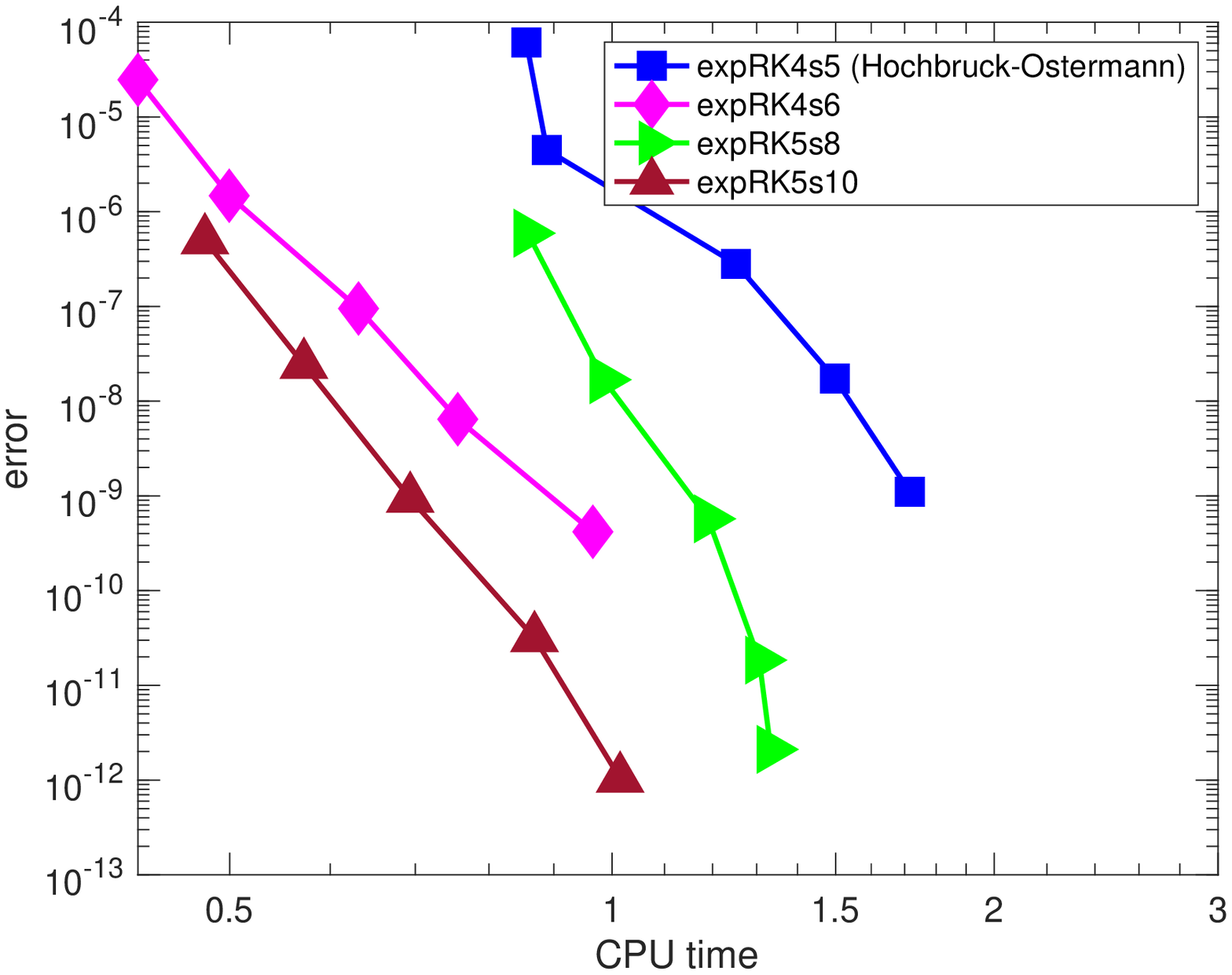,width=0.47\linewidth,clip=}  
\end{tabular}
\caption{\label{fig6.1} Order plots (left) and total CPU times (right) of $\mathtt{expRK4s5}$, $\mathtt{expRK4s6}$ , $\mathtt{expRK5s8}$, and $\mathtt{expRK5s10}$ when applied to \eqref{example1}. The global errors at time $t=1$ are plotted as functions of the number of time steps (left) and the total CPU time in second (right). For comparison, straight lines with slopes 4 and 5 are added.}
\end{figure}

\begin{example}\label{ex2}\rm
(\emph{A nonlinear Schrödinger equation \cite{Cazenave1989,Berland2005}}): 
We consider the following one-dimensional nonlinear Schrödinger (NLS) equation with periodic boundary conditions 
\vspace{-2mm}
\begin{equation} \label{example2}
\begin{aligned}
\mathtt{i} \frac{\partial \Psi(x,t)}{\partial t}   & = - \frac{\partial^2 \Psi(x,t)}{\partial x^2} + \big(V(x)+ \lambda |\Psi(x,t)|^2\big)\Psi(x,t), \\
 \Psi(-\pi,t) &= \Psi(\pi,t), \q t \ge 0\\
\Psi(0,t) &=\Psi_0 (x),   \q x \in [-\pi, \pi]
\end{aligned}
\end{equation}
where the potential function $V(x)=\dfrac{1}{1+\sin^2 (x)}$, the initial condition $\Psi_0 (x)=\ee^{\sin(2x)}$, and the constant
 $\lambda =1$  (see  \cite{Berland2005}).
\end{example}
\emph{Spatial discretization}: For this example, we use a discrete Fourier transform $\mathcal{F}$  with $ND=128$ modes, leading to a  mildly stiff  system of the form  \eqref{eq1.1} with
\begin{equation} \label{eq6.3}
\begin{aligned}
A&= \text{diag}(-\mathtt{i} k^2),\  k= -\frac{ND}{2}+1, \cdots, \frac{ND}{2}= -63, \cdots, 64 \\
g(t,u)&=-\mathtt{i} \mathcal{F}((V(x)+ \lambda |\mathcal{F}^{-1}(u)|^2)\mathcal{F}^{-1}(u).
\end{aligned}
\end{equation}
Next, we integrate this system on the time interval $[0, 3]$ with constant step sizes, corresponding to the number of time steps  
$N= 64,         128$,        $ 256 ,        512,        1024$.
 Since the exact solution  $ \Psi(x,t)$ of  \eqref{example2} is unknown,  a reliable reference solution is computed by the stiff solver $\mathtt{ode15s}$ with $ATOL=RTOL=10^{-14}$. Again, the time integration errors are  measured in a discrete maximum norm at the final time $t= 3$.

As seen from the  two double-logarithmic diagrams in Figure~\ref{fig6.2}, we plot the accuracy of the four employed integrators ($\mathtt{expRK4s5}$, $\mathtt{expRK4s6}$ , $\mathtt{expRK5s8}$, and $\mathtt{expRK5s10}$) as functions of the number of time steps (left) and the total CPU time (right). The left  digram clearly indicates that the two new integrators $\mathtt{expRK4s6}$ and $\mathtt{expRK5s10}$ achieve their corresponding expected orders 4 and 5. While $\mathtt{expRK5s10}$ is a little more accurate than $\mathtt{expRK5s8}$,  $\mathtt{expRK4s6}$ is much more accurate than $\mathtt{expRK4s5}$ for a given same number of time steps, meaning that it can take much larger time steps while achieving the same accuracy. Moreover,  the right precision digram displays the efficiency plot indicating that both $\mathtt{expRK4s6}$ and $\mathtt{expRK5s10}$ are much faster than their counterparts $\mathtt{expRK4s5}$ and $\mathtt{expRK5s8}$, respectively.
More specifically,  a similar story is observed: for lower accuracy requirements, say error $\sim 10^{-7}$,  the new fourth-order method $\mathtt{expRK4s6}$  is the most efficient, whereas for error  $\sim 10^{-8}$ or tighter the new fifth-order method $\mathtt{expRK5s10}$  is the most efficient. 


\begin{figure}[H]
\centering
\begin{tabular}{cc}
\epsfig{file=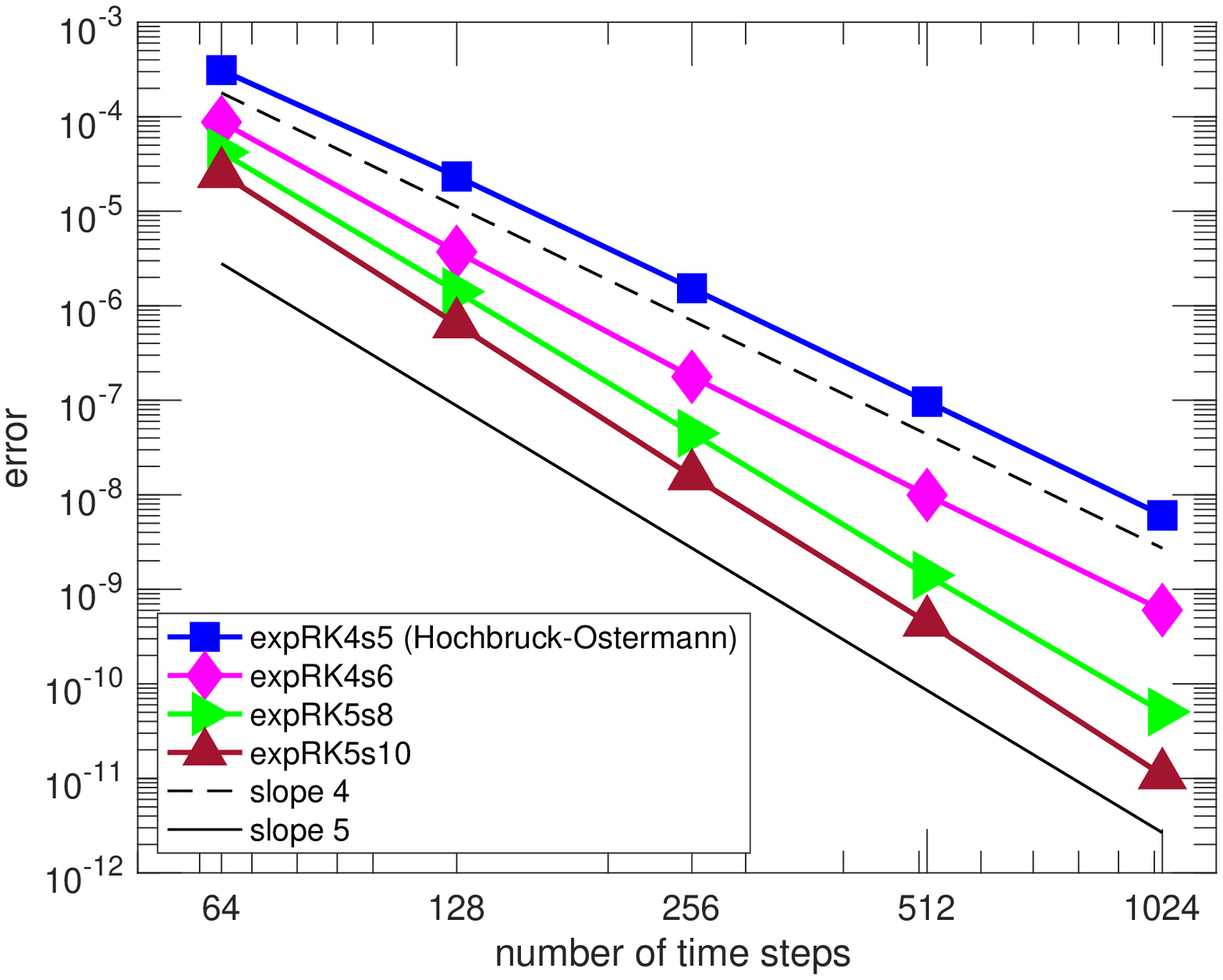,width=0.47\linewidth,clip=}
\hspace{2mm}
\epsfig{file=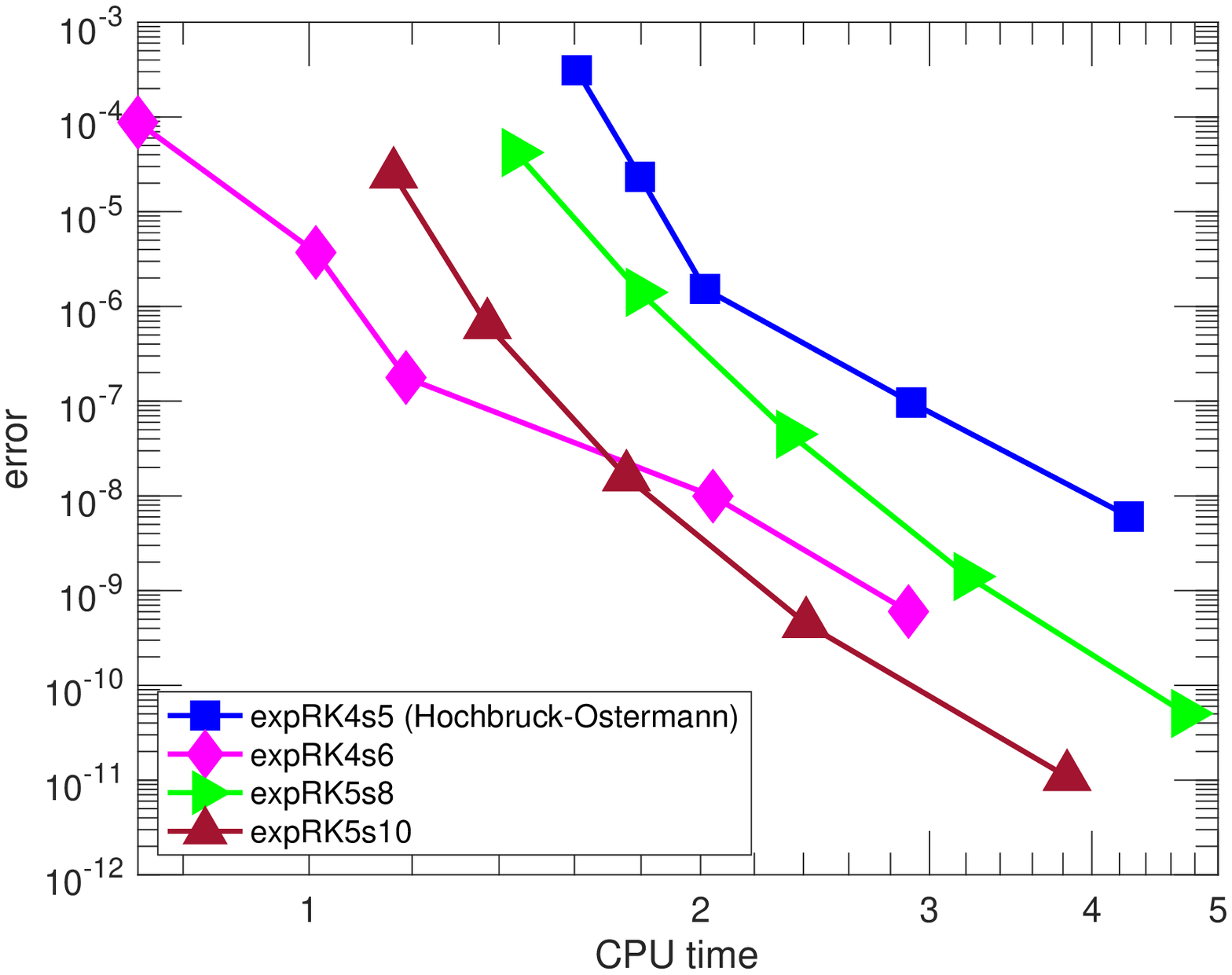,width=0.47\linewidth,clip=}  
\end{tabular}
\caption{\label{fig6.2} Order plots (left) and total CPU times (right) of $\mathtt{expRK4s5}$, $\mathtt{expRK4s6}$ , $\mathtt{expRK5s8}$, and $\mathtt{expRK5s10}$ when applied to Example~\ref{ex2}. The errors at time $t=3$ are plotted as functions of the number of time steps (left) and the total CPU time in second (right). For comparison, straight lines with slopes 4 and 5 are added.}
\end{figure}

\begin{example}\label{ex3}\rm
(\emph{A 2D  Gray--Scott model \cite{Gray1984,Berland2007}}): 
Consider the following two-dimensional reaction-diffusion equation--the Gray--Scott equation model, for $u = u(x, y, t ),  \ v= v(x, y, t)$  on the square $\Omega=[0, L]^2$, (here, we choose $L=1.5$) subject to periodic boundary conditions
\begin{equation} \label{example3}
\begin{aligned}
\frac{\partial u}{\partial t}& =d_u \Delta u - uv^2+ \alpha(1-u),\\
\frac{\partial v}{\partial t}&= d_v \Delta v +uv^2 -(\alpha+\beta)v,
\end{aligned}
\end{equation}
where $\Delta$ is the Laplacian operator, the diffusion coefficients  $d_u=0.02, \ d_v=0.01$, and the bifurcation parameters $\alpha=0.065,\ \beta=0.035$. The initial conditions are Gaussian pulses 
\[
u(x,y,0)=1-e^{-150\big((x-L)^2+ (y-L)^2\big)}, \ v(x,y,0)=e^{-150\big((x-L)^2+ 2(y-L)^2\big)}.
\]
\end{example}
\emph{Spatial discretization}: For this example, we use standard second order finite differences using 150 grid points in each direction with mesh width $\Delta x = \Delta y = L/150$. This gives a stiff system of the form  \eqref{eq1.1}.

The system is then solved on the time interval $[0, 2]$ using constant step sizes.
In the absence of an analytical solution of \eqref{example3},  a high-accuracy reference solution is computed using the $\mathtt{expRK4s6}$ method with a sufficient small time step.  Errors are  measured in a discrete maximum norm at the final time $t= 2$.

In Figure~\ref{fig6.3},  using the same number of time steps   $N=32,  64,         128$,        $ 256 ,        512,        1024$, we again display the order plots of the taken integrators. One can see  that $\mathtt{expRK4s6}$ is much more accurate than $\mathtt{expRK4s5}$ and  $\mathtt{expRK5s10}$ is slightly more accurate than $\mathtt{expRK5s8}$. 

\begin{figure}[H]
\centering
\epsfig{file=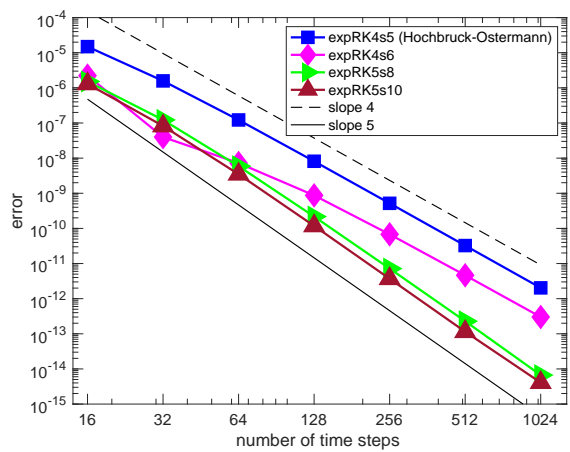,width=0.6\linewidth,clip=}
\caption{\label{fig6.3} Order plots of $\mathtt{expRK4s5}$, $\mathtt{expRK4s6}$ , $\mathtt{expRK5s8}$, and $\mathtt{expRK5s10}$ when applied to Example~\ref{ex3}. The errors at time $t=2$ are plotted as functions of the number of time steps. For comparison, straight lines with slopes 4 and 5 are added.}
\vspace{-2mm}
\end{figure}
In Figure~\ref{fig6.4}, we display the efficiency plot for which the time step sizes were chosen for each integrator to obtain about  same error thresholds $10^{-i}, \ i= 5,\cdots, 11$
(The corresponding number of time steps for each integrator are displayed in  Table~\ref{table1}. As seen, given about the same level of accuracy, the new methods use smaller steps than the old ones of the same order, meaning that they can take larger step sizes).
Again, $\mathtt{expRK4s6}$ is much faster than $\mathtt{expRK4s5}$ and it is interesting that this new fourth-order method turns out to be the most efficient (although for error thresholds tighter than $10^{-11}$ the new fifth-order method $\mathtt{expRK5s10}$ seems to become the most efficient).

\begin{figure}[H]
\centering
\epsfig{file=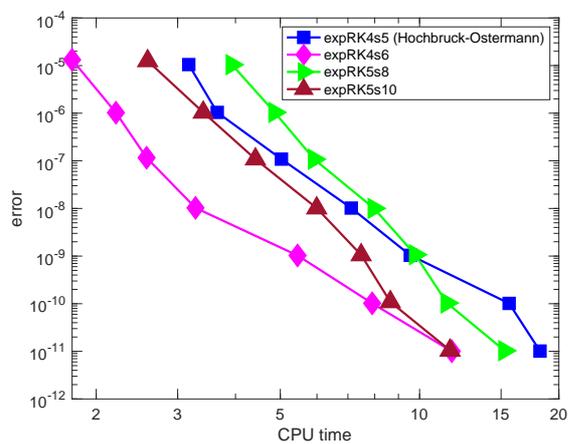,width=0.6\linewidth,clip=}  
\caption{\label{fig6.4} Total CPU times of $\mathtt{expRK4s5}$, $\mathtt{expRK4s6}$ , 
$\mathtt{expRK5s8}$, and $\mathtt{expRK5s10}$ when applied to Example~\ref{ex3}. 
The time step sizes were chosen in such a way that each integrator achieves about the same error thresholds $10^{-i}, \ i= 5,\cdots, 11$.
The errors at time $t=2$ are plotted as functions of the total CPU time (in second).}
\vspace{-8mm}
\end{figure}
\setlength{\extrarowheight}{2 pt}
\begin{table}[H]
\centering
\begin{tabular}{ |c|c|c|c|c|c|c|c| }
\hline
\multirow{2}{*}{Method} &\multicolumn{7}{c|}{Error threshold vs. Number of time steps} \\
 \cline{2-8}
&$10^{-5}$ & $10^{-6}$ &$10^{-7}$ &$10^{-8}$&$10^{-9}$&$10^{-10}$&$10^{-11}$ \\
\hline
 $\mathtt{expRK4s5}$  & 18 & 36 & 66 & 121& 215& 385& 685 \\
$\mathtt{expRK4s6}$ &10&19& 28& 46& 122& 230& 420 \\
 $\mathtt{expRK5s8}$ &7&  18&33& 57&  92& 149& 238   \\
 $\mathtt{expRK5s10}$  &8&  17& 30& 51&  82& 130&208 \\
\hline
\end{tabular}
\caption{\label{table1} The number of time steps taken to achieve about the same error thresholds $10^{-i}, \ i= 5,\cdots, 11$.}
\vspace{-2mm}
 \end{table}

The numerical results presented on the three examples above clearly confirm  the advantage of constructing parallel stages expRK methods  based on Theorem~\ref{th3.1}, leading to more efficient and accurate methods $\mathtt{expRK4s6}$ and $\mathtt{expRK5s10}$.
 \begin{acknowledgements}
The author would like to thank Reviewer~1 for  the valuable comments and helpful suggestions.
He would like also to thank the National Science Foundation, which supported this research under award NSF DMS--2012022.
 \end{acknowledgements}

\bibliographystyle{spmpsci}      
\bibliography{references} 

\end{document}